\documentclass[11pt]{article}
\usepackage{authblk}
\usepackage{comment}

\makeatletter
\newcommand\subsubsubsection{\@startsection{paragraph}{4}{\z@}{-2.5ex\@plus -1ex \@minus -.25ex}{1.25ex \@plus .25ex}{\normalfont\normalsize\bfseries}}
\newcommand\subsubsubsubsection{\@startsection{subparagraph}{5}{\z@}{-2.5ex\@plus -1ex \@minus -.25ex}{1.25ex \@plus .25ex}{\normalfont\normalsize\bfseries}}
\makeatother

\usepackage{amsfonts}
\usepackage{amsmath,amssymb}
\usepackage{amscd}
\usepackage{amsthm}
\usepackage{fancyhdr}
\usepackage{color}
\usepackage{hyperref} 
\usepackage{tikz-cd}
\usepackage{txfonts}
\newcommand{\we}{\wedge}

\newcommand{\pp}[2]{\frac{\partial#1}{\partial#2}}

\newtheorem*{theoremA}{Theorem \hypertarget{thm.a}{A}}
\newtheorem*{theoremAA}{Theorem A}
\newtheorem*{theoremB}{Theorem B}
\newtheorem*{theoremC}{Theorem \hypertarget{thm.c}{C}}
\newtheorem*{theoremCC}{Theorem C}
\newtheorem*{theoremD}{Theorem D}

\usepackage{almacros} 

\title{$E$-structures and almost regular Poisson manifolds}

\author[1]{Alfonso Garmendia}
\affil[1]{Max Planck institute for mathematics, Bonn, Germany.
Email: 
\texttt{garmendia@mpim-bonn.mpg.de
}
\footnote{Both authors are supported by the Spanish State Research Agency MCIU/AEI / 10.13039/501100011033/ FEDER, UE., through the Severo Ochoa and María de Maeztu Program for Centers and Units of Excellence in R\&D (project CEX2020-001084-M) and the Spanish State Research Agency grants reference PID2019-103849GB-I00 of AEI / 10.13039/501100011033 and PID2023-146936NB-I00 funded by MICIU/AEI/
10.13039/501100011033 and, by ERDF/EU.  The authors are also partially supported by the AGAUR project 2021 SGR 00603 Geometry of Manifolds and Applications, GEOMVAP.} }

\author[2]{Eva Miranda}
\affil[2]{ Laboratory of Geometry and Dynamical Systems \& SYMCREA, Department of Mathematics, EPSEB, Universitat Polit\`{e}cnica de Catalunya-IMTech
in Barcelona and
\\ CRM Centre de Recerca Matem\`{a}tica, Campus de Bellaterra
Edifici C, 08193 Bellaterra, Barcelona
 \footnote{  Eva Miranda is supported by the Catalan Institution for Research and Advanced Studies via an ICREA Academia Prize 2021 and by the Alexander Von Humboldt Foundation via a Friedrich Wilhelm Bessel Research Award. Eva Miranda is also supported by the Spanish State
Research Agency, through the Severo Ochoa and Mar\'{\i}a de Maeztu Program for Centers and Units
of Excellence in R\&D (project CEX2020-001084-M). } 
Email:
\texttt{ eva.miranda@upc.edu}
 }
\begin{document}

\date{}

\maketitle
\abstract{In recent years, $b$‑symplectic manifolds have emerged as important objects in symplectic geometry. These manifolds are Poisson manifolds that exhibit symplectic behavior away from a distinguished hypersurface, where the symplectic form degenerates in a controlled manner. Inspired by this rich landscape, $E$-structures were introduced by Nest and Tsygan in \cite{NT2} as a comprehensive framework for exploring generalizations of $b$-structures. This paper initiates a deeper investigation into their Poisson facets, building on foundational work by \cite{MS21}. We also examine the closely related concept of almost regular Poisson manifolds, as studied in \cite{AZ17}, which reveals a natural Poisson groupoid associated with these structures.

In this article, we investigate the intricate relationship between $E$-structures and almost regular Poisson structures. Our comparative analysis not only scrutinizes their Poisson properties but also offers explicit formulae for the Poisson structure on the Poisson groupoid associated to the $E$-structures as both Poisson manifolds and singular foliations. In doing so, we reveal an interesting link between the existence of commutative frames and Darboux-Carathéodory-type expressions for the relevant structures.

}
 
 \section{Introduction}

Singular symplectic manifolds and more concretely $E$-symplectic manifolds have been the object of intense investigation in the last years. The notion of $E$-symplectic manifold appeared for the first time in the article \cite{NT2} where the authors discuss these structures as a generalization of $b$-symplectic structures. Given a locally free finitely generated $C^\infty(M)$-module $E$ of vector fields an $E$-structure as done in \cite{MS21} is a structure on a natural vector bundle $\EA$ associated to $E$.  For any $E$-symplectic manifold there is natural dual object associated to it which is a Poisson bivector. $E$-structures are also closely related to foliation theory, in particular, locally free finitely generated $C^\infty(M)$-modules of vector fields are the same as a special class of singular foliations called ``almost regular'' (as in \cite{P01}) and the associated vector bundle $\EA$ is naturally an almost injective Lie algebroid. 


This paper investigates interactions among three notions; almost regular foliations, and two generalizations of $b$-symplectic manifolds: $E$-symplectic manifolds and almost regular Poisson structures. In section \ref{sec.def} we will properly introduce these objects, nevertheless, let us give an intuitive introduction in the following paragraphs.

An almost regular foliation of rank $k\in \mathbb{N}$ is a singular foliation, meaning a subset of vector fields $E\subset \mathfrak{X}(M)$ closed under the Lie bracket on a manifold $M$, characterized by the existence of $k$ vector fields $X_1,\dots,X_k\in \mathfrak{X}(U)$ in a neighbourhood $U$ of a point $p\in M$, such that locally $E|_U=\left<X_1,\dots,X_k\right>_{\gi(U)}$. Moreover, these vector fields must be $\mathcal{C}^\infty(U)$-independent, nevertheless some of them can vanish at certain points in $U$ allowing for singularities.

In $\mathbb{R}^2$, for instance, the almost regular foliation $E=\{X\in \mathfrak{X}(M) \st X \, \text{is tangent to } Z=\{0\}\times \KR\}$ is generated (globally) by exactly two vector fields, namely, $x\dex$ and $\dey$. In this case the singular foliation is called $b$-foliation where $b$ is reminiscent of $b$-manifolds introduced by Melrose to address the index theorem on manifolds with boundary $Z$. 
$b$-Manifolds have been the main characters in the theory of deformation quantization on symplectic manifolds with boundary as observed by Nest and Tsygan \cite{NT1}. These generalized singular symplectic manifolds have been denoted in the literature under several names $b$-symplectic or $log$-symplectic manifolds and are our main source of inspiration in this article. For them, $Z$ will no longer refer to the boundary but rather to a submanifold often called \emph{critical set}.

For the generalizations of $b$-symplectic manifolds: {\it \textbf{$E$-symplectic manifolds} involve symplectic structures on almost regular foliations}, and {\it \textbf{almost regular Poisson manifolds}, extending $b$-symplectic manifolds as Poisson manifolds, where the symplectic foliation is an almost regular foliation.}

To exemplify these generalizations, a $b$-symplectic manifold exhibits two distinct almost regular foliations: one defined by the $b$-tangent bundle and the other by its symplectic foliation, illustrated further in the ensuing example.

\begin{ex}\label{ex.b.poi.edge}
    Let us use $M=\KR^2$ and $Z=\{(0,y)\in \KR^2 \}$ 
    \begin{enumerate}
        \item  The submodule $E=\{X\in \gx(M) \st X|_Z\subset TZ\}$ is an almost regular foliation with associated vector bundle being $\EA \cong\KR^2\times M$ and an anchor map $\rho:\EA\fto TM$ given by 
        $$\rho(a,b,x,y)=xa\dex|_{(x,y)} + b\dey |_{(x,y)}.$$
         Let $X,Y:M\fto \EA$ be the  canonical constant sections generating $\gs(\EA)$, i.e. $X(p)=(e_1, p)$ and $Y(p)=(e_2,p)$ for all $p\in M$, where $e_1,e_2$ are the canonical basis of $\KR^2$. Let $\alpha, \beta\in \gs(\EA^*)$ be the dual basis of $X,Y$.
        
         The closed and non-degenerate two form $\omega=\alpha\wedge\beta$ gives the $E$-symplectic manifold $(M,\EA,\rho,\omega)$. Its associated bivector is $\pi_{\omega}=X\wedge Y\in \gs(\EA^{\wedge 2})$.
        
        This structure induces an almost regular Poisson structure $\pi^\sharp:=\rho\circ \pi_\omega^\sharp\circ \rho^*\colon T^*M\fto TM$ in $M$. This structure satisfies:
        $$\pi^\sharp(dx):=\rho(\pi_{\EA}^\sharp(\rho^*(dx)))=\rho(\pi_{\EA}^\sharp(x\alpha))=\rho(x Y) =x\dey,$$
        $$\pi^\sharp(dy):=\rho(\pi_{\EA}^\sharp(\rho^*(dy)=\rho(\pi_{\EA}^\sharp(\beta))=\rho(-X) =-x\dex,$$
        therefore the Poisson bivector in $M$ is $\pi=x(\dex\wedge \dey)$.

    \item Starting with the almost regular Poisson structure $\pi=x(\dex\wedge \dey)$ in $M$, the symplectic foliation is the following almost regular foliation:
    $$E'=\pi^\sharp(\gw(M))=\left<-x\dex, x\dey\right>=\{X\in \gx(M) \st X|_Z=0\}.$$
     Its vector bundle is given by $\EEA\cong\KR^2\times M$ with anchor:
     $$\rho(a,b,x,y)=xa\dex|_{(x,y)} + xb\dey |_{(x,y)}.$$
     Let $X',Y'\subset \gs(\EEA)$ the canonical constant sections. Let $\alpha',\beta'\in \gs(\EEA^*)$ be the dual sections of $X',Y'$.
    
  The map $\lambda\colon T^*M \to \EEA$ is given by $\lambda(dx) = Y'$ and $\lambda(dy) = -X'$. Its dual  $\lambda^*\colon \EEA^*\fto TM$ is given by the formulas $\lambda^*(\alpha')=-\dey$, $\lambda^*(\beta')=\dex$. Then, there is a closed 2-form $\omega_{\pi} = x(\alpha' \wedge \beta')$ that is symplectic in an open dense subset of $M$ and induces the Poisson structure on $M$.
  
    \end{enumerate}
\end{ex}

The example above illustrates the correspondence between $b$-symplectic manifolds and $b$-Poisson manifolds explained in \cite{GMP14}. Moreover, it also highlights some differences between the $b$-foliation and the symplectic foliation.

Another interesting feature is that the symplectic foliation of the Poisson structure associated with a $b$-symplectic manifold with compact leaves on the critical set is endowed with an \emph{edge structure}. Edge structures were studied by Fine \cite{fine}, motivated by their connections to twistor theory. An edge structure is associated with a submanifold $Z$ that is also a fibration, characterized as follows: edge vector fields are not only tangent to $Z$ at points in $Z$, but they remain tangent to the fibers of the fibration.

Given a $b$-symplectic manifold $(M,Z)$ with compact singular set $Z$ and an embedded symplectic leaf in $Z$, the symplectic foliation defined by the Poisson structure provides an example of such a structure (see Ex. \ref{ex.edge}).

In \cite{fine}, an edge structure naturally arises in the twistor space of $\mathbb H^4$. Let $(Z, J)$ denote the twistor space of $\mathbb H^4$ with the Eells–Salamon almost complex structure. In twistor coordinates $(x, y_i, z_i)$, this almost complex structure blows up as $x$ approaches zero. Edge geometry provides a framework to treat singularities in $J$ as smooth up to the boundary. The edge structure considered in \cite{fine} is associated with the fibration $\partial Z \to S^3$ given by the twistor projection. Another edge structure under consideration in \cite{fine} is associated to a compact Riemann surfaces with boundary, where the projection of the boundary as fibration is the identity map (a so-called $0$-structure).

These examples motivate exploring the relationships between $E$-symplectic and almost regular Poisson structures, leading to the following results:

\begin{theoremAA}
    Any $E$-symplectic structure on $M$ induces a Poisson structure. If $E$ is regular or of maximal rank, the Poisson structure is regular or almost regular.

Conversely, any almost regular Poisson manifold $(M, \pi)$ with an almost regular symplectic foliation $E$ has a closed $E$-form of degree 2 that is $E$-symplectic on an open dense subset of $M$.
\end{theoremAA} 

The letter $E$ is used in the articles \cite{NT1} and \cite{MS21}, which inspired this work. The theorem above implies that both $E$-symplectic and almost regular Poisson structures with symplectic foliation $E$ are $E$-structures.

 For any almost regular foliation, there is a Lie groupoid integrating it (confer \cite{P01}). This  Lie groupoid (whose set of arrows is a finite-dimensional manifold) is a model for the infinite-dimensional group of flows of elements in $E$.

Moreover, as a consequence of Lie bi-algebroid theory \cite{AZ17}, the $E$-structures from $E$-symplectic and almost regular Poisson manifolds generate a unique multiplicative Poisson structure in the groupoid. This structure is defined using abstract constructions and studying its properties can be challenging.

In this paper, we present concrete results for this Poisson structure, including explicit formulas in some cases. We would like to emphasize the following results:

\begin{theoremB}\label{theorem.b}
    The multiplicative Poisson structure $\hat{\pi}$ in $\CG$ given by $\pi$ is a regular Poisson structure. Moreover, $\pi_\CG$ is completely characterized by the equations $\bs_*(\pi_\CG)=-\pi$ and $\bt_*(\pi_\CG)=\pi$.

\end{theoremB}

 This statement says that the Poisson structure at the groupoid level does not have singularities. Moreover, Corollary \ref{thm.sym.real.form} states that if we have explicit formulas for $\bs$ and $\bt$ in a chart, we can explicitly describe this Poisson structure within that chart. Androulidakis and Skandalis constructed such explicit formulas for $\bt$ and $\bs$ in certain charts in the paper \cite{AS09}, where the groupoid is locally diffeomorphic to open sets $\CU \subset \mathbb{R}^k \times M$. Here, the source map is the projection to $M$, $k$ is the number of vector fields generating the almost regular foliation, and the target map is determined by following the flows of linear combinations of these generators. A summary of this result is provided in Section \ref{sec.hol.grpd}.

We also consider specific cases as the symplectic integration of $E$-structures including $b$-, $b^m$-, and elliptic structures, culminating in the following theorem, whose proof and precise statement are provided in Section \ref{sub.sub.section.f.poi}.
\begin{theoremCC} Let $(M,\pi)$ any Poisson manifold of dimension $2+2k$ such that $\pi$ is locally written as:
$$\pi= \left(f(x)\dex\wedge\dey \right)+ \pi_0(z,w)$$
for $f$ with discrete zeros, $\pi_0$ dual to a symplectic form in $\KR^{2k}$, $x,y$ coordinates in $\KR$ and $z,w$ coordinates in $\KR^k$.\\

The Poisson manifold $(M,\pi)$ has a symplectic integration that, near the identity, looks like $\KR^4\times\KR^{2k}\times\KR^{2k}$  with Poisson structure:
$$\pi_\CG\scalebox{.75}{$(a,\!b,\!x,\!y,\! z'\!,\!w'\!,\!z,\!w)$}\!=\! \left(\scalebox{1.1}{$\dea\!\wedge\!\dey$} +\alpha\scalebox{.8}{$(a,\!x)$}\,\scalebox{1.1}{$\deb\!\wedge\!\dex$}+\scalebox{1.35}{$\frac{b}{a}$}(1\!-\!\alpha\scalebox{.8}{$(a,\!x)$})\,\scalebox{1.1}{$\deb\!\wedge\!\dey$} -f(x)\,\scalebox{1.1}{$\dex\!\wedge\!\dey$}\right)+ \pi_0\scalebox{.8}{$(z'\!,\!w')$}  -\pi_0\scalebox{.8}{$(z,\!w)$}\,,$$
where, $a=b=0$, $z'=z$, $w'=w$ represents the identity bisection, the source and target maps are given explicitly, and the function $\alpha(a,x)$ is determined by the resolution of a specific ODE, and satisfies $\alpha(0,x)=1$.
\end{theoremCC}

As a corollary of the theorem above, we provide a global symplectic realization for the $b^m$-Poisson structure $x^m\dex\wedge\dey+\pi_0$ in $\KR^2\times\KR^{2k}$ with $\pi_0$ symplectic.

  The integration in Theorem \hyperlink{thm.c}{C} corresponds to a Poisson groupoid integrating a bi-algebroid that is locally diffeomorphic to a holonomy groupoid and to a source-simply connected symplectic integration, as in \cite{CF03}. Details will be provided in section \ref{sec.int.ar}, \ref{sub.sub.section.b.sym} and \ref{sec.ar.poi}.

Moreover, in subsection \ref{sub.sub.section.cosymp}, we examine the special case of cosymplectic manifolds. In this context, we obtain a cosymplectic groupoid whose symplectization serves as the symplectic integration of the Poisson manifold underlying the original cosymplectic manifold. Notably, the cosymplectic groupoid studied here extends the work initiated in \cite{GMP11} but differs from the investigations conducted by Rui Loja Fernandes and David Iglesias Ponte \cite{FI23}. For the cosymplectic groupoids in this article, in contrast with \cite{FI23}, the identity bisection is not part of a single symplectic leaf; rather, it generates every other element in the groupoid through Hamiltonian flows.

The Poisson groupoid of a general $E$-symplectic manifold is described by the following theorem:

\begin{theoremD}
      The multiplicative Poisson structure $\pi_\CG$ in $\CG$ given by $\omega$ is a Poisson structure given by an $(\bs^{-1}(E))$-symplectic structure $\hat{\omega}$.   
\end{theoremD}

The theorem above states that the Poisson structure $\pi_\CG$ is of the same ``type'' as $\pi$. A well-known case is that of $b^m$-symplectic manifolds: the groupoid integrating the $b^m$-foliation is also $b^m$-symplectic. However, the groupoid integrating the Hamiltonian vector fields is symplectic, as established in Theorem \hyperlink{thm.c}{C}. The desingularization procedure developed in \cite{GMW19} provides a way to connect these integrating groupoids, as discussed in Section~\ref{sec:b.desing}. Additionally, we examine the existence of Darboux-type normal forms for $E$-symplectic structures and explore the relationship between the existence of commutative frames and Darboux-Carathéodory normal forms.

\subsection*{Organization of this paper}
In Section \ref{sec.obj}, we review and summarize the objects of interest in this paper, with a particular focus on symplectic structures on almost regular foliations, also referred to as $E$-symplectic manifolds and almost regular Poisson manifolds. We provide examples and discuss key aspects of the groupoids that integrate these objects. \

Section \ref{sec.3} presents the main results of the paper. In Subsection \ref{sec.e.ar}, we state one of our principal results, Theorem A, which establishes a connection between symplectic structures on almost regular foliations, almost regular Poisson manifolds, and Lie bialgebroids. In Subsection \ref{sec.poi}, we discuss the integration of Lie bialgebroids into Poisson groupoids. Subsection \ref{sec.int.ar} offers a local description of the Poisson structure for almost regular Poisson manifolds, leading to Theorem B. Specifically, for $b^m$-symplectic manifolds and various other almost regular Poisson manifolds, we derive the symplectic groupoid integration, which is presented as Theorem \hyperlink{thm.c}{C}.\

In Section \ref{sec.ar.poi}, we outline a strategy for obtaining the symplectic groupoid integration of an almost regular Poisson manifold, applicable to cases such as $b^m$-symplectic, elliptic, and cosymplectic manifolds, among others. Finally, in Section \ref{sec.int.e}, we explore the case of $E$-symplectic manifolds, providing Theorem D and discussing its relation to Darboux forms, including the Darboux-Carathéodory theorem.

In the appendix, we explicitly describe the groupoid composition and inverse for the holonomy and source simply connected groupoid of a foliation in special charts, considering commutative frames.

\subsection*{Acknowledgments} We thank Camilo Angulo, Joel Villatoro and Marco Zambon for useful conversations regarding the organization of the paper and examples for the integration of almost regular Poisson manifolds.




\section{The objects}\label{sec.obj}
\subsection{Definitions}\label{sec.def}

We consider special cases of Poisson manifolds and of singular foliations. To be self-contained we include here both definitions:

\begin{defn}
    A Poisson manifold is a smooth manifold $M$ with a bivector field $\pi\in \gx^2(M)$ such that $[\pi,\pi]=0$ where the bracket is the canonical extension of the Lie bracket from vector fields to multi-vector fields given by Schouten and Nijenhuis.
\end{defn}

\begin{defn}
    A singular foliation (as in \cite{AS09}) is a locally finitely generated $\gi(M)$-submodule $E\subset \gx(M)$ such that $[E,E]\subset E$.
\end{defn}

These two represent geometric structures that control the dynamics on a manifold. On the one hand, any function $H \in \gi(M)$ on a Poisson manifold $(M,\pi)$ yields a vector field  $\pi(dH,-)\in \gx(M)$ governing the dynamics in $M$.

On the other hand, we can see the vector fields of a singular foliation $E$ as governing the dynamics on $M$ (alternatively they can also be seen as symmetries). 

The special cases we want to study are:

 \begin{defn} An almost regular foliation (coined by Debord in \cite{P01}) is a singular foliation $E\subset \gx(M)$ such that $E$ is a locally free module (or projective module). 
 \end{defn}

 \begin{rk}
     We will also refer to almost regular foliations as $E$-structures as done in \cite{MS21}.
 \end{rk}

 \begin{defn}
     An almost regular Poisson manifold (as in \cite{AZ17}) is a Poisson manifold $(M,\pi)$ such that the set of the Hamiltonian vector fields $E=\pi^\sharp(\gw(M))$ defines an almost regular foliation.
 \end{defn}
 
Any almost regular foliation $E$ has a vector bundle $\EA\fto M$ and a vector bundle map $\rho\colon \EA\fto TM$, called anchor which is injective on an open dense subset of $M$ and such that $E=\rho \gs(\EA)$. This implies that $E\cong \gs(\EA)$ and therefore the Lie bracket in $E$ induces canonically a Lie bracket on $ \gs(\EA)$. This means that $\EA$ has a more interesting structure, making it fit in the following definition.

\begin{defn} A {Lie algebroid} on a manifold $M$ is a vector bundle $A\fto M$ with a vector bundle map $\rho\colon A\fto TM$, called the anchor,  and a Lie bracket $[-,-]$ on $\gs(A)$ satisfying for any $X,Y\in \gs(A)$ and $f\in \gi(M)$ the following formula:
$$[X,fY]=df(\rho(X)) Y + f[X,Y].$$
\end{defn}

\begin{defn} An {ai-algebroid} (almost injective) is a Lie algebroid $A\fto M$ such that its anchor map $\rho$ is injective in an open dense subset of $M$. If ${\rm rnk}(M,A)=\dim(M)$ it is said to be of \textbf{full rank}.
\end{defn}

 The following statement proved in \cite{AZ13} and appearing here as a proposition shows a 1-1 correspondence, up to isomorphism, between almost regular foliations and ai-algebroids.

\begin{prop}
    Let $M$ be a manifold and $E\subset \gx(M)$ a singular foliation, the following three statements are equivalent.
    \begin{enumerate}
        \item $\dim(E/I_x E)$ has constant finite dimension $k\in \KN$ for all $x\in M$, where $I_x=\{ f\in \gi(M) \st f(x)=0\}$
        \item $E$ is an almost regular foliation such that, the minimal amount of generators for $E$ at any $x\in M$ is the constant $k\in \KN$,
        \item there is an ai-algebroid $\EA\fto M$ of rank $k\in \KN$ with $\rho(\gs(\EA))=E.$
    \end{enumerate}     
\end{prop}
\begin{defn}
    Let $E$ be an almost regular singular foliation. The \textbf{rank of $E$} is the rank of its ai-algebroid $\EA$. We say that \textbf{$E$ is of full rank} if $\EA$ is of full rank.
\end{defn}
 Every Poisson manifold $(M,\pi)$ inherently possesses a Lie algebroid structure given by $\pi^\sharp \colon T^*M \fto TM$ and therefore a singular foliation $E=\pi^\sharp(\gw(M))$ called the symplectic foliation.

Considering the reciprocal relation, if we begin with an almost regular foliation $E$ whose ai-algebroid is $\EA$ and a section $\pi \in \gs(\EA^{\wedge 2})$ with $[\pi,\pi]=0$, then $\pi$ induces a Poisson structure on $M$. A case of study in this article is when $\pi$ comes as the dual of a symplectic (closed and non-degenerate) form $\omega\in \gs((\EA^*)^2)$. 

Let us paraphrase the previous statement in a different view. Given an $E$-structure, by the Serre-Swan theorem, there is an \textbf{$E$-tangent bundle} ${}^E\!TM(={}^E\!A)$, whose sections (locally) are sections of $E$, and an \textbf{$E$-cotangent bundle} ${^E}TM^* := ({^E}TM)^*={}^E\!A^*$. We will refer to the global sections of $\wedge^p ({^E}TM^*)$ as \textbf{$E$-forms of degree $p$}, and denote the space of all such sections by ${^E}\Omega^p(M)$.

As $E$ satisfies the involutivity condition $[E, E] \subseteq E$, there is a differential $d: {^E}\Omega^p(M) \rightarrow {^E}\Omega^{p+1}(M)$ given by the Cartan-type (or Leibnitz-type) formula:

$$d\omega\scalebox{.85}{$(V_0, \dots,\! V_p)$} = \!\sum_i (-1)^iV_i\left(\omega\scalebox{.85}{$\left(V_0, \dots,\! \hat{V_i}, \dots,\! V_p \right)$} \right) + \sum_{i < j} (-1)^{i+j}\omega\scalebox{.85}{$\left([V_i, V_j], V_0, \dots,\! \hat{V}_i, \dots,\! \hat{V}_j, \dots,\! V_p\right)$}\,,$$
where the hat as in $\hat{V_i}$ represents a missing element.\\

The cohomology of this complex is the \textbf{$E$-cohomology} ${^E}H^*(M)$.\\ 

A closed non-degenerate $E$-form of degree 2 is called an \textbf{$E$-symplectic form}, and the triple $(M, E, \omega)$ is referred to as an \textbf{$E$-symplectic manifold}.

\begin{defn}
    Let $E$ be an almost regular foliation on the manifold $M$. An \textbf{$E$-symplectic} structure (as in as in \cite{MS21}) is a symplectic form $\omega\in \gs((\EA^*)^2)$. 
\end{defn}

For obstructions to the existence of $E$-symplectic structures, we suggest consulting \cite{K20} where the author discusses obstructions for symplectic structures on Lie algebroids. In the next subsection, we will also give several examples that inspired our work.

\subsection{Some motivating examples}

\begin{ex}\label{ex.ar.various} We start providing a list of examples of almost regular foliations:

\begin{enumerate}
    \item Let $M$ be a manifold and $Z$ a codimension 1 submanifold.
    \begin{enumerate}
        \item $M$ with $E^b$ being all the vector fields tangent to $Z$ gives the ai-algebroid ${}^{E^b}\!A={}^b TM$, called the $b$-tangent bundle.
        \item $M$ with $E^0$ being all the vector fields vanishing along $Z$ gives  ai-algebroid ${}^{E^0}\!A={}^0 TM$, called the $0$-tangent bundle. This vector bundle is used to study special cases of elliptic operators as can be seen in \cite{MU}.
        \item Let $\pi:Z\fto N$ be a submersion. On $M$ let $E^e\subset \gx(M)$ be all vector fields tangent to the fiber of $\pi$ in $Z$. $E^e$ gives the ai-algebroid ${}^{E^e}\!A={}^eTM$, called the edge-tangent bundle \cite{fine}. 
\end{enumerate}

    \item $\KR^2$ with $E$ generated by the rotations and the radial Euler vector field generate the elliptic Lie algebroid, which is an ai-algebroid. This algebroid has been studied by many authors, in particular its co-homology can be seen in 
    \cite{AW22}.
\end{enumerate}
\end{ex}

\begin{ex}(Regular foliations are almost regular foliations)
    An important example of a non-full rank $E$-manifold is the one of regular foliations. A regular foliation is a subbundle $F\subset TM$ such that $[\gs(F),\gs(F)]\subset \gs(F)$. Clearly, $F$ is a sub Lie algebroid of $TM$ and therefore an (almost) injective Lie algebroid.

    Moreover, any Lie algebroid with injective anchor is isomorphic to a regular foliation.
\end{ex}

\begin{ex}
    Given two almost regular singular foliations $(M_1,E_1)$ and $(M_2,E_2)$, the cartesian product $\left(\,M_1\times M_2\,,\, E_1\oplus E_2\subset \gx(M_1\times M_2)\,\right)$ is an almost regular foliation with ai-algebroids given by the Cartesian product as well.  
\end{ex}

\begin{ex}\label{ex.edge} Given any $b$-symplectic manifold the set of the Hamiltonian vector fields $E:=\gx_{\rm Ham}(M)=\pi^\sharp(\gw(M))$ is an almost regular foliation, as we showed in detail in Example \ref{ex.b.poi.edge}. \\
Moreover, if we ask that its singular subset $Z$ to be compact and the existence of an embedded compact symplectic leaf, then by \cite[Theorem 19]{GMP11} and \cite[Theorem 49]{GMP14}, we can describe $E$ as all the vector fields in $M$ tangent to a fibration $Z\fto S^1$ over a circle \footnote{In the original result $Z$ is a mapping torus $L\times_f I\fto S^1$ where $L$ is a compact symplectic leaf and $f$ is a symplectomorphism.} .This implies that $E$ is an edge structure as seen in part 1.(c) of example \ref{ex.ar.various} and in \cite{fine}.
\end{ex}

\subsection{The groupoid of an almost regular foliation}\label{sec.hol.grpd}

Here we want to highlight some key elements in the construction of the holonomy and source-simply connected groupoids associated to a almost regular foliation. We will give here a short introduction but for a more complete one, we refer to section 4 and 5 of \cite{GV22}.\\

The holonomy groupoid associated with any singular foliation is constructed in \cite{AS09}. Groupoids are often studied using paths (see \cite{CF03, CF04}). The groupoid can be described as paths in $E$ modulo holonomy, yielding an effective action. The source simply connected integration is introduced in \cite{GV22}. In \cite{P01}, it is proven that for almost regular foliations, this groupoid is a finite-dimensional smooth manifold, and thus a Lie groupoid.
\\

For any almost regular foliation $E$ on a manifold $M$:
\[{\rm Hol}(E)={\rm Paths}(E)/{holonomy} \y \CG(E)={\rm Paths}(E)/{(E-homotopy)}\]

These two groupoids are locally diffeomorphic and can be described using local charts that are then glued using a process similar to \cite[Theorem 3.4]{gualtieri}. These local charts are charts near the identity of both groupoids and any other element is reached by composition. Let us remark here bellow how any of the two groupoids will look locally near the identity:
\begin{itemize}
  
    \item For each $x\in M$, near the identity element at $x$, the holonomy (and source-simply connected) groupoid $\CG$ is locally diffeomorphic to an open neighbourhood $\CU\subset \KR^k\times M$ of $(0,x)$, where $k$ is the rank of $E$.
    
    \item Let $U$ the projection into $M$ of $\CU$. It is possible to find a diffeomorphism from the mentioned above such that the unit elements coincide with the zero section $\iota\colon U\fto \KR^k\times U$, the source $\bs\colon \CU\fto U$ with the projection to $U$ and the  target $\bt\colon \CU\fto U$ with the time 1 flow:
    $$\bt(v_1,\cdots,v_k,u)=\Phi^{v_1\rho(X_1)+\cdots+v_k(X_n)}_1(u),$$
    where $X_1,\cdots, X_k$ are any local sections of $\EA$ that are linearly independent in $U$.

\end{itemize}

To put these facts into context, we want to mention that in the original article \cite{AS09}, the quadruple $(\CU,\iota,\bs,\bt)$ are called path holonomy bisubmersions.  The holonomy and source simply connected groupoid can be constructed by ``gluing'' these objects. \\

Under this description, we know that $\CG$ looks like $\CU \subset \mathbb{R}^k \times M$ and how the identity bisection, the source, and the target will appear. Nevertheless, there is no clear way to explicitly write the groupoid composition and inverse. With an extra assumption, we have a way to express them using the following lemma.

\begin{thm}
\label{lem.comm.loc}
    If $X_1,\cdots,X_k$ commute under the Lie bracket, then the groupoid composition, and inverse of the holonomy and source simply connected groupoid $\CG$ is described in $\CU\subset \KR^k\times M$ by the normal addition and negative elements in $\KR^k$.
\end{thm}

The proof of this theorem follows from Proposition \ref{prop:comm.chart.grpd} in the appendix. Both this theorem and Proposition \ref{prop:comm.chart.grpd}
are well-established facts about Lie algebroids, often considered folklore knowledge. We include them for the sake of self-containment, as we could not find a clear reference for them.\\

In the following sections, the notions of symplectic and Poisson structures compatible with the groupoid structure will be important, and a complete description of the composition will be key. This is why commutative frames will play a central role in the explicit characterization of the Poisson structure.

\section{Almost regular Poisson manifolds and $E$-symplectic manifolds}\label{sec.3}
In this section we study almost regular Poisson and $E$-symplectic structures as two cases of Lie bi-algebroids. We also study the canonical Poisson structure induced on the source simply connected groupoid integrating $\EA$. 

We want to remark that, by \cite[Theorem 4.3]{AZ17} its holonomy groupoid, which is a (discrete) quotient of the source simply connected groupoid integrating it (and therefore it is locally diffeomorphic), has a Poisson structure too.

\subsection{Relations between $E$-symplectic manifolds and almost regular Poisson manifolds}\label{sec.e.ar}

By the definition of almost regular Poisson manifold $(M,\pi)$ the symplectic foliation $E$ is almost regular and for its ai-algebroid $\EA$ there is a unique surjective morphism $\lambda: T^*M\fto \EA$ making the following commutative diagram:
    \[\begin{tikzcd}
     & \EA \arrow[dr,"\rho"] & \\
        T^*M \arrow[ur, "\lambda"] \arrow[rr, "\pi^\sharp"]  & & TM \\ 
    \end{tikzcd}\]
    
Remark that the existence of $\lambda$ is guaranteed by the following argument: The map at the level of sections $\pi^\sharp: \Omega(M)\rightarrow E$  is surjective and $\rho: \Gamma \EA \rightarrow E$ is $\gi(M)$-module isomorphism at the level of sections, then there is a unique map $\lambda: \Omega(M)\rightarrow \Gamma(\EA)$  which is surjective and $\gi(M)$-linear. This linearity, yields a surjective vector bundle map.

\begin{theoremA}\label{lem:esym.alps}\label{theorem.a}
\begin{enumerate} 
    \item Let $E$ be an ai-singular foliation with ai-algebroid $\EA\fto M$. Every $E$-symplectic structure induces canonically a Poisson bi-vector field $\pi\in \gx^2(M)$. If the $E$ is of full rank then the Poisson manifold is almost regular.
    \item Let $(M,\pi)$ be an almost regular Poisson manifold with symplectic foliation $E$, and ai-algebroid $\EA$. There is a closed form $\omega_{\pi}\in \gw^2(\EA)$ that is symplectic in an open dense subset of $M$. If $E$ is of full rank then $\EA\cong T^* M$ and $\omega_{\pi}\in \gw^2(\EA)$ is $\pi\in \gx^2(M)\cong \gw^2(T^*M)$.
\end{enumerate}
\end{theoremA}
\begin{proof}
    \begin{enumerate}
        \item Let $\omega\in \gw^2(\EA))$ be an $E$-symplectic structure. Because the 2-form $\omega$ is non-degenerate it can be dualized to a bivector $\pi_\omega\in \gs(\EA^{\wedge 2})$. The following map defines a bivector in $M$:

        \[\begin{tikzcd}
            T^*M \arrow[r,"\rho^*"] & (\EA)^* \arrow[r, "\pi_\omega^\sharp"] & \EA \arrow[r,"\rho"] & TM .
        \end{tikzcd}\]
This bivector is Poisson as a consequence of $\omega$ being closed (see Lemma \ref{lem.d.pi}).\\

If $A$ is of full rank, the dual map $\rho^*$ (which in local coordinates corresponds to the transpose) is almost injective and therefore $E=(\rho\circ\pi^\sharp_\omega\circ\rho^*)(\gw(M))$ is an almost regular foliation (see example \ref{ex.not.ar} for when it is not almost regular). 
    \item Because the bivector $\pi$ is by definition skew-symmetric, the dual of its musical morphism $\pi^\sharp\colon T^*M\fto TM$ satisfies $(\pi^\sharp)^*=-\pi^\sharp$. This implies the existence of the following commutative diagram of vector bundles:
    \[\begin{tikzcd}
     & \EA \arrow[dr,"\rho"] & \\
        T^*M \arrow[ur, "\lambda"] \arrow[rr, "\pi^\sharp"] \arrow[dr,swap, "-\rho^*"] & & TM \\
         & \EA^* \arrow[ur,swap, "\lambda^*"] 
    \end{tikzcd}\]
    The map $\lambda$ is surjective by construction; therefore, $\lambda^*$ is injective. Consequently, there exists a map $\omega_{\pi}^\sharp: \EA \to \EA^*$ making the diagram commute. The map $\omega_{\pi}^\sharp$ is skew-symmetric as $\pi$ is, i.e., $(\omega_{\pi}^\sharp)^* = -\omega_{\pi}^\sharp$, and therefore it corresponds to a 2-form $\omega_{\pi} \in \Gamma(\wedge^2 \EA^*)$. This $\omega_{\pi}$ is closed as a consequence of $\pi$ being Poisson (see Lemma \ref{lem.d.pi}). Moreover, $\omega_{\pi}$ is non-degenerate when $\rho$ is injective, i.e., on an open dense subset of $M$.
    If $E$ is of full rank then $\lambda$ is an isomorphism. 
    \end{enumerate}
       
\end{proof}
The ``full rank'' condition in part 1. of the lemma above is necessary as the reader can see in the following (counter) example:
\begin{ex}\label{ex.not.ar}
    Let $M=\KR^4$ with coordinates $x,y,z,w$ and the ai-singular foliation $E$ given by the vector fields $X_1=x\dex+w\dew$ and $X_2=y\dey+z\dez$. This is is given by the ai-algebroid $\EA=\KR^2\times M$ with anchor $\rho\colon \EA\fto TM$ sending $(e_i,p)\mapsto X_i(p)$ for $e_1,e_2$ the canonical basis of $\KR^2$ and $p\in M$.

    The dual map is given by $\rho^*\colon T^* M\fto \EA^*$ with 

    $$\begin{matrix}
        dx\mapsto xe_1^* & dy\mapsto ye_2^*\\
        dz\mapsto ze_2^* & dw\mapsto we_1^*
    \end{matrix},$$
    This image is not a projective submodule of $\gs(\EA^*)$, because the minimal amount of generators is $2$ everywhere except when one of the coordinates $x,y,z,w$ are zero.

  In $\EA$, we can have the (constant) symplectic bivector $\pi=e_1 \wedge e_2$, which corresponds to the symplectic form $\omega=e_1^* \wedge e_2^*$. The Poisson structure induced in $M$ is given by the assignment
    $$\begin{matrix}
        dx\xrightarrow{\rho^*} xe_1^*\xrightarrow{\omega^\sharp} -xe_2\xrightarrow{\rho} -xX_2 & dy\xrightarrow{\rho^*}ye_2^*\xrightarrow{\omega^\sharp} ye_1\xrightarrow{\rho} y X_1\\
        dz\xrightarrow{\rho^*} ze_2^*\xrightarrow{\omega^\sharp} ze_1 \xrightarrow{\rho} z X_1& dw\xrightarrow{\rho^*} we_1^*\xrightarrow{\omega^\sharp} -we_2\xrightarrow{\rho} -wX_2
    \end{matrix},$$
    which gives the (not almost regular) Poisson bivector:
    $$\pi:=-xy \dex\wedge \dey -xz \dex\wedge \dez +yw \dey\wedge \dew +zw\dez\wedge \dew=X_2\wedge X_1.$$
\end{ex}

We suggest that the reader refer to Example \ref{ex.b.poi.edge} for a detailed illustration of these relations.

To understand the integration of the ai-algebroids for $E$-symplectic and almost regular Poisson we can fit them inside the theory of Lie bi-algebroids. Let us recall this definition, see \cite{AD90} and the references within.

\begin{defn} A Lie bi-algebroid is a pair of Lie algebroids $A$ and $A^*$ such that the $A^*$-derivation $d_{A^*}\colon \gs^{\wedge k}(A)\fto \gs^{\wedge k+1}(A)$ satisfies Leibnitz with the bracket of $A$, i.e.
	$$d_{A^*}[X,Y]_A = [d_{A^*}X,Y]_{A}\pm [X,d_{A^*}Y]_A.$$
\end{defn}
The bracket of a Lie algebroid can be extended using Leibnitz identity to the exterior algebra $\gs(\wedge A)$ and it defines a De-Rham type differential on the dual exterior algebra $d\colon \gs(\wedge A^*)\fto \gs(\wedge^{+1} A^*)$.

 An $E$-symplectic manifold $(M,\EA,\omega)$ give us the Lie bi-algebroids $(\EA^*,\EA, d=[\omega,-])$ and its dual $(\EA,\EA^*, d^*=[\pi_\omega,-])$. Moreover, an almost regular Poisson manifold $(M,\pi)$ with ai-singular foliation $E'$ give us another Lie bialgebroid $(\EEA^*,\EEA,d=[\omega_{\pi},-])$ and its dual, which differential $d^*$ is not necessarily given by a bivector in $\EEA$. 

The following lemma can be seen as an exercise (in fact, it is exercise 12.1.19 in \cite{MK}). We are stating it here for the sake of self-containment.

\begin{lem}\label{lem.d.pi}
    Let $A\fto M$ be a Lie algebroid and $\pi\in \gs(A^2)$ such that $[\pi,\pi]=0$. Then $A^*$ is a Lie algebroid with:
    $$\rho'\colon A^*\fto TM; \alpha\mapsto \rho(\pi^\sharp(\alpha))$$
    $$[\alpha,\beta](X)=\rho'(\alpha)(\beta(X))-\rho'(\beta)\alpha(X)+[\pi,X](\alpha,\beta),$$
    The de Rham derivation in $A=(A^*)^*$ is exactly the graded Lie bracket $d_{A^*}=[\pi,-]$. Moreover, $\pi^\sharp\colon A^*\fto A$ is a Lie algebroid morphism.

    If $\omega\in \gs((A^*)^2)$ is non-degenerate with dual $\pi_\omega\in \gs(A^2)$ then $[\pi_\omega,\pi_\omega]=0$ if and only if $d\omega=0$. 
\end{lem}

   


 \subsection{Review on Poisson groupoids for $E$-structures}\label{sec.poi}
 
 Because of theorem 4.1 by Mackenzie and Xu in \cite{MX00}, for any Lie bi-algebroids $(A,A^*)$, with $A$ integrable by a source simply connected Lie groupoid $\CG$, there is a unique Poisson structure $\pi_\CG$ in $\CG$ which makes it into a Poisson groupoid with Lie bialgebroid $(A,A^*)$. By what we mentioned in subsection \ref{subsection:holgrpd}, the Lie algebroid associated to an $E$-type submodule is integrable. Therefore, any $E$-symplectic manifold and almost regular Poisson structure has an associated source simply connected Poisson groupoid. Let us review the definition of Poisson groupoid due to \cite{W88}:

\begin{defn}
     A multivector vield $\pi\in \gx^{\wedge k}(\CG)$ in a Lie groupoid $\CG$ is \textbf{multiplicative} if the corresponding map $\pi_\CG^\sharp\colon \wedge^{k-1}T^*\CG\fto T\CG$ is a a groupoid morphism.
 \end{defn}

 \begin{defn}
     A \textbf{Poisson groupoid} is a Lie groupoid with a \textbf{multiplicative} Poisson bivector $\pi_\CG\in \gx^{\wedge 2}(\CG)$.
 \end{defn}

 The condition of $\pi_\CG$ being multiplicative can be described with many equivalent statements, we suggest the reader to check proposition 3.3 of the original paper for the integration \cite{MX00} and the notes \cite{LMX11}. With the definition given here the proof of the  following lemma is straightforward:

\begin{lem}\label{lem.loc.poi.grpd}
    If $\CG$ is source connected, the multiplicative bivector $\pi_\CG$ is completely characterized by its values near the identity or equivalently $\pi_\CG^\sharp$ is characterized by its values near the zero section of $M$ in $T^* M$.
\end{lem}
\begin{proof}
   As $T^*\CG$ is source-simply connected, any morphism, and in particular $\pi_\CG^\sharp$ is characterized by its values near the identity bisection $\iota(T^*M)\subset T^*\CG$. Moreover, its linearity characterizes it in a neighbourhood of the zero section of $T^*M$. 
\end{proof}
 
The following result, confirms the result of Lemma \ref{lem.loc.poi.grpd}, as it reconstructs the Poisson structure from its Lie bi-algebroid (which only encodes information in the neighbourhood of $\iota(M)$). It is a consequence of the original construction of the Poisson groupoid for a Lie bi-algebroid done in \cite{MX00}. We recommend also checking section 3.2 of the Lecture notes \cite{LMX11} showing this construction.

\begin{thm}\label{thm.e.symp.grpd0} Let $(M,A, d^*=[\pi,-])$ be a Lie bi-algebroid for some $\pi\in \gs(A^{\wedge 2})$ and $\CG\soutar M$ the source simply connected groupoid integrating $A$. The pullback Lie algebroid $A_\CG=\bt^{!}A=\bt^*A {}_\rho \!\times_{d\bt} T\CG$ (do not confuse with the pullback vector bundle $\bt^*A$) and the vector bundle $A_\CG'=\bt^*A\otimes \bs^* A$ are isomorphic (as vector bundles) under the isomorphism:
$$\varphi: A'_\CG\fto A_\CG, (v,w,g)\mapsto (v,v\circ 0_g + 0_g \circ w^{d\tau}).$$

Moreover, the section $\pi_\CG'=\pi\oplus -\pi\in \gs((A_\CG')^{\wedge 2})$ and the section $\varphi^{\wedge 2}(\pi_\CG')\in \gs(A_\CG^{\wedge 2})$ gives the unique multiplicative Poisson structure $\pi_\CG$ of \cite{MX00} by the formula $\pi_\CG=\rho^{\wedge 2}\varphi^{\wedge 2}\pi_\CG'\in \gx^2(\CG)$, more explicitly: 
$$\forall g\in \CG \,\,\,\, \Rightarrow \,\,\,\, \pi_\CG|_g= \left((\pi|_{\bt(g)}) \circ_{d\mu^{\wedge 2}} 0_g^{\wedge 2} \right)- \left(0_g^{\wedge 2} \circ_{d\mu^{\wedge 2}}(\pi|_{\bs(g)})^{d\tau^{\wedge 2}}\,\right)\in T_g^{\wedge 2}\CG .$$
\end{thm}

As a consequence of this theorem we obtain,

\begin{cor}\label{cor.uniqueness.of.pi}
    Let $(M,A,\rho:A\fto TM, d^*=[\pi,-])$ be a Lie bi-algebroid for some $\pi\in \gs(A^*)$ and $(\CG\soutar M,\pi_\CG)$ its a Poisson groupoid then $\rho^{\wedge 2}\pi=\bt_*(\pi_\CG)$ and $\rho^{\wedge 2}\pi=-\bs_*(\pi_\CG)$.
\end{cor}

Even though, after seeing these definitions, one expects to be able to write the Poisson structure in $\CG$ explicitly, the formulas depend on the groupoid structure and more precisely on the composition $\circ$. This composition is in many cases quite difficult to work with. In this article, we seek alternatives to explicitly obtain $\pi_\CG$.\\

In the following, we will describe this structure as accurately as possible for a neighbourhood of the zero section of $T^* M$. This is equivalent to expressing  $\pi_\CG$ near the identity bisection $\iota(M)\subset \CG$.\\

Let us write the bivector as $\pi=\sum_{i\in I} f_i (v_i \wedge w_i)$  for a finite set $I$ and some sections $v_i,w_i\in \gs(A)$ (every bivector in $A$ has this form locally). For any $x\in M$ and $V,W\in T_{\iota_x}\CG$ such that $d\bt(W)=d\bs(V)$ one can prove that the following formulas hold: 
$$V\circ_{d\mu} W= V+W-d\iota(d\bs(V)),$$
$$\Rightarrow V^{d\tau}=-V+d\iota(d\bt(V)+d\bs(V)).$$ 

Using the formula above, theorem \ref{thm.e.symp.grpd0} and the splitting $T\CG|_M= A\otimes TM$ it is possible to rewrite $\pi_\CG$ on the identity bisection as follows:
	\begin{equation}\label{eq.poi.id}
	\pi_{\CG_{\iota(x)}}=\sum_{i\in I} f_i(x) \left(\,\, v_i(x)\!\wedge\!\rho(w_i(x)) \,+\,\rho(v_i(x))\!\wedge\! w_i(x) \,-\, \rho(v_i(x))\!\wedge\!\rho(w_i(x))\,\right)
	\end{equation}
	for all $x\in M$.

We can use the equation \ref{eq.poi.id} to relate the types of singularities of the Poisson groupoid associated to a $E$-symplectic manifold or an almost regular Poisson structure, with the starting symplectic or Poisson structure. We will also try to expand this formula to a neighbourhood of the identity in special cases.

\subsection{The Poisson groupoid integrating the Hamiltonian vector fields for almost regular Poisson manifolds}\label{sec.int.ar}

Let $(M,\pi)$ be an almost regular Poisson manifold, with almost regular foliation $E=\pi^\sharp(\gw(M))$ and ai-algebroid $\EA$. Let $\lambda$ and $\rho$ be the maps $T^*M\xrightarrow{\lambda} \EA \xrightarrow{\rho} TM$ given in the proof of Theorem \hyperlink{thm.a}{A}, $\CG$ the source simply connected groupoid integrating $\EA$, $\CG^*$ the source simply connected groupoid integrating $\lambda^*\colon \EA^*\fto TM$ and $\omega_\pi\in \gw^{ 2}(\EA)=\gs((\EA^*)^{\wedge 2})$ the unique element such that $(\lambda^*\wedge \lambda^*) (\omega_\pi)=\pi$.

By construction the map $\lambda$ is surjective therefore the map $\lambda^*$ is injective. This means that $\CG^*$ is the monodromy groupoid of a regular foliation, which is well described in many books and articles such as \cite{H} and \cite{MM03}. Moreover, regular foliations have commutative frames (def \ref{def:comfrm}) and we can use the local charts in Proposition \ref{prop:comm.chart.grpd} to describe its groupoid and Poisson structure.\\

\begin{lem}\label{lem.comm.loc.poi} For any $p\in M$ let $U'$ be a neighbourhood, $\alpha_1,\cdots,\alpha_k\in \gs((\EA)^*)|_U$ be a commutative frame,
\begin{itemize}
    \item Let $\CG^*$ the source simply connected groupoid integrating $\EA^*$ i.e. the fundamental groupoid of the regular foliation given by $\EA^*$.
    \item Let $ \KR^k\times M \supset \CU \soutar U'$ be the groupoid given by Lemma \ref{lem.comm.loc} with diffeomorphism $\phi:\CU\fto \CG^*$.
\end{itemize}
Then
\begin{enumerate}
    \item There is only one Poisson structure $\pi_{\CG^*}$ in $\CG^*$ satisfying $(\bt\times\bs)_*\pi_{\CG^*}=\pi\oplus -\pi$, for the map $(\bt\times\bs)\colon \CG^*\fto M\times M$.
    \item Writing the bivector $\omega_\pi\in \gs((\EA^*)^{\wedge 2})$ locally at $p\in M$ as $\omega_\pi|_U=\sum_{i,\!j} f_{i,\!j} (\alpha_i\wedge \alpha_j)$ for $f_{i,\!j}$ smooth functions, the Poisson structure $\hat{\pi}_\phi$ in $\CU$ given below is the only one satisfying $\phi_*\hat{\pi}_\phi=(\pi_{\CG^*})|_{\CU}$:
    $$\hat{\pi}_\phi\scalebox{.85}{$(u)$}=\!\sum_{i,\!j} \left(f_{i,\!j}\scalebox{.85}{$(\bt(u))$}\!-\!f_{i,\!j}\scalebox{.85}{$(\bs(u))$}\right) \left(\scalebox{1.2}{$\frac{\partial}{\partial z_i}\!\wedge\! \frac{\partial}{\partial z_j}$}\right)\, +\, f_{i,\!j}\scalebox{.85}{$(\bs(u))$} \left(\scalebox{1.2}{$\frac{\partial}{\partial z_i}$}\!\wedge\!\rho\scalebox{.85}{$(\alpha_j)$}\!+\!\rho\scalebox{.85}{$(\alpha_i)$}\!\wedge\!\scalebox{1.2}{$\frac{\partial}{\partial z_j}$}\!-\!\rho\scalebox{.85}{$(\alpha_i)$}\!\wedge\!\rho\scalebox{.85}{$(\alpha_j)$}\right)\,,$$
     where $z_i$ are the coordinates of $\KR^k$ of $\CU\subset \KR^k\times M$. Moreover,
    the image of $((\bt\times\bs)(\CU),(\bt\times\bs)_*\hat{\pi}_\phi)$ is a Poisson submanifold of $(M\times M, \pi\oplus-\pi)$.
    \item $\pi_{\CG^*}$ and $\hat{\pi}_\phi$ are the Poisson structures given by Theorem \ref{thm.e.symp.grpd0}.
\end{enumerate}
\end{lem}

\begin{proof}
    Under the assumptions of Proposition \ref{prop:comm.chart.grpd} there is $d\bt(\frac{\partial}{\partial z_i})=\rho(\alpha_i)$ and $d\bs(\frac{\partial}{\partial z_i})=0$. By Lemma \ref{lem.regular.pair} there is at most one bivector $\hat{\pi}_\phi$ in $\CU$ satisfying $(d\bs)^2(\hat{\pi}_\phi)=-\pi$ and $(d\bt)^2(\hat{\pi}_\phi)=\pi$. One can check that the bivector given in the proposition satisfies this property.

    Moreover, because of Corollary \ref{cor.uniqueness.of.pi}, we have $\phi_*\hat{\pi}_\phi=(\pi_{\CG^*})|_{\CU}$ and using the groupoid structure on $\CU$ one can check that it is multiplicative.
\end{proof}

\begin{rk}
   In the proof above, even though uniqueness implies that $\hat{\pi}_\phi$ must be multiplicative, we recommend that the reader verify its multiplicativity using the explicit composition in $\CU$ (given by lemma \ref{lem.comm.loc}).
\end{rk}

Now that we have an understanding of the Poisson structure for $\CG^*$, we will describe the Poisson structure in $\CG$.
\begin{thm}\label{thm.uniq.pi}
    The multiplicative Poisson structure $\pi_\CG$ in $\CG$ given by $\pi$  is completely characterized by the equations $d\bs_*(\pi_\CG)=-\pi$ and $d\bt_*(\pi_\CG)=\pi$.
\end{thm}

\begin{proof}
   Let \(\varphi\colon \CG \to \CG^*\) be the Poisson morphism whose derivative near the identity is given by the map
\[
\omega_{\pi}^\sharp\colon \EA \to (\EA)^*,
\]
which is an almost injective Lie algebroid morphism. This implies that \(\varphi\) is a local diffeomorphism on an open dense subset of \(\CG\). Consequently, there is at most one Poisson structure \(\pi_\CG\) on \(\CG\) that pushes forward to \(\pi_{\CG^*}\), where \(\pi_{\CG^*}\) is the unique Poisson structure on \(\CG^*\) satisfying
\[
d\bs_*(\pi_{\CG^*}) = -\pi \quad \text{and} \quad d\bt_*(\pi_{\CG^*}) = \pi,
\]
as stated in Lemma \ref{lem.comm.loc.poi}.

\end{proof}
 
  The map $\varphi\colon \CG\fto \CG^*$  depends on each case, nevertheless, we know its derivative near the identity is $\omega_{\pi}^\sharp:\EA \fto (\EA)^*$.  Then we can write the Poisson structure in $\CG$ in the identity bisection regardless of the formulas defining $\varphi$, leading to the following result.

\begin{cor}\label{thm.sym.real.form}
    For any almost regular Poisson manifold, with almost symplectic bisection written locally as  $\omega_\pi|_U=\sum_{i,\!j} f_{i,\!j} (\alpha_i\wedge \alpha_j)$ for smooth functions $f_{i,\!j}\in \gi(U)$ and a commutative frame $\alpha_1,\cdots,\alpha_k\in \gs(\EA^*)$ then $\EA$ has local generators $X_1,\cdot,X_k$ satisfying $\alpha_i(X_j)= \delta_{ij}$ and $\rho(X_s)=\sum_j (f_{s,j}-f_{j,s}) \rho(\alpha_j)$. The Poisson bivector on its holonomy groupoid restricted to the identity bisection is given by:
    $$(\pi_\CG)|_{\iota(p)}= \sum_j {X_j}(p)\wedge \rho(\alpha_j)|_p - \sum_{i,\!j} f_{i,\!j}(p) (\rho(\alpha_i)|_p\wedge \rho(\alpha_j)|_p),$$ 
    Using the splitting $T\CG|_M= \EA \oplus_M TM$.\\
    This implies that $(\CG,\pi_\CG)$ is a regular Poisson structure of rank $2k$.\\
    
    Moreover, If $\EA$ is of full rank then $\pi_\CG$ is symplectic and $(\CG,\pi_\CG)$ is the symplectic integration of the Poisson manifold $(M,\pi)$.\\
\end{cor}
 
The importance of the above theorem is that it gives a formula for the bivector in the identity bisection and formulas that work for more general settings that can be extended to a neighbourhood of the identity in specific cases as we will do in the next section when we can write the map $\varphi$ explicitly. 

\begin{proof}

Let $\CG,\CG^*$ be the source-simply connected Lie groupoids integrating $A,A^*$ respectively. Let $p,\CU', U'$ as in Lemma \ref{lem.comm.loc.poi} the Poisson structure is
$$\phi^*\hat{\pi}\scalebox{.85}{$(u)$}=\!\sum_{i,\!j} \left(f_{i,\!j}\scalebox{.85}{$(\bt(u))$}\!-\!f_{i,\!j}\scalebox{.85}{$(\bs(u))$}\right) \left(\scalebox{1.2}{$\frac{\partial}{\partial z_i}\!\wedge\! \frac{\partial}{\partial z_j}$}\right)\, +\, f_{i,\!j}\scalebox{.85}{$(\bs(u))$} \left(\scalebox{1.2}{$\frac{\partial}{\partial z_i}$}\!\wedge\!\rho\scalebox{.85}{$(\alpha_j)$}\!+\!\rho\scalebox{.85}{$(\alpha_i)$}\!\wedge\!\scalebox{1.2}{$\frac{\partial}{\partial z_j}$}\!-\!\rho\scalebox{.85}{$(\alpha_i)$}\!\wedge\!\rho\scalebox{.85}{$(\alpha_j)$}\right)\,,$$
If we restrict to $M$ and use the canonical isomorphism $T\CG^*|_M\cong A^* \oplus TM$ we have:
$$\phi^*\hat{\pi}(\iota(p))=\sum_{i,\!j} f_{i,\!j}(p) \left(\alpha_i\wedge \rho(\alpha_j) + \rho(\alpha_i)\wedge \alpha_j - \rho(\alpha_i)\wedge \rho(\alpha_j)\right),$$
$$\phi^*\hat{\pi}(\iota(p))=\sum_{i,\!j} (f_{i,\!j}-f_{j,i})(p) \left(\alpha_i\wedge \rho(\alpha_j) - \rho(\alpha_i)\wedge \rho(\alpha_j)\right),$$
The map $\varphi_{\omega}\colon \CG\fto \CG^*$ satisfies that $d\varphi|_M = \omega^{\sharp}\oplus Id: A\oplus TM\fto A^*\oplus TM$. Then the Poisson structure in $\CG|_M$ is 
$$\pi_\CG|_{\iota(p)}=\varphi^*\omega = \sum_j {X_j}(p)\wedge \rho(\alpha_j)|_p - \sum_{i,\!j} f_{i,\!j}(p) (\rho(\alpha_i)|_p\wedge \rho(\alpha_j)|_p).$$

If we assume the anchor on the dual is surjective (and therefore bijective), then $X_i,\rho(\alpha_i)$ are all the coordinates in $U$ (and $\CG$ is even dimensional). Then the bivector $\pi_\CG|_M$ is non-degenerate, therefore it is symplectic near $M$ and using that it is multiplicative, and $\CG$ is source connected then it must be symplectic everywhere.

\end{proof}

 As a consequence of Theorem \ref{thm.uniq.pi} and Corollary  \ref{thm.sym.real.form} we obtain,

\begin{theoremB}
 
    The multiplicative Poisson structure $\hat{\pi}$ in $\CG$ given by $\pi$ is a regular Poisson structure. Moreover, $\pi_\CG$ is completely characterized by the equations $\bs_*(\pi_\CG)=-\pi$ and $\bt_*(\pi_\CG)=\pi$.

\end{theoremB}

The regularity of $(\CG, \hat{\pi})$ has already been proven in \cite{AZ17}. The significance of this theorem is that it provides explicit formulas for $\hat{\pi}$. If the source and target maps can be explicitly written for $\CG$, then it becomes possible to describe $\pi_\CG$. In the next section, we will specifically address the $b$-symplectic case.

 \subsection{The Poisson groupoid integrating the  Hamiltonian vector fields for a $b^m$-symplectic manifold.}\label{sub.sub.section.b.sym}

Consider a $b$-manifold $(M,Z)$. The set of vector fields tangent to $Z$ up to a certain order, say $m$, gives rise to a generalization of the $b$-tangent bundle, known as the $b^m$-tangent bundle. In particular, the entire construction of $b$-manifolds can be generalized to include $b^m$-symplectic manifolds, as Scott did in \cite{Scott} and later, with a global description in \cite{FPW23}. The $b^m$-counterparts have become interesting objects of study in quantization \cite{GMW21} and in investigating the convexity of toric actions \cite{GMW18}. A key observation, due to Guillemin, Miranda, and Weitsman, is that it is possible to desingularize $b^m$-symplectic manifolds \cite{GMW19}. In this section, desingularization plays a crucial role, as it allows us to identify the integrating groupoid as the limit of a desingularized construction, which can be interpreted as the integration of an almost regular Poisson manifold. Thus, all the pieces of the puzzle come together seamlessly.

In this section, we only consider $\KR^2$. Still, the results are valid for any $b^m$-symplectic manifolds since any $b^m$-symplectic manifold is locally isomorphic to the Cartesian product of a symplectic manifold with the examples given here.

\subsubsection{The case of $b$-symplectic manifolds}

In $\KR^2$ consider the $b$-Poisson bivector field $\pi=x\dex \wedge \dey$. This defines the singular foliation $E$ generated by $x\dex$ and $x\dey$. Let $E$ be the Lie algebroid of $E$. The Lie algebroids $E$ and $T^*M$ are isomorphic. Let $\CG(E)$ the Lie groupoid integrating $E$.  The groupoid $\CG(E)$ integrating $E$ is isomorphic to a neighbourhood $U\subset \KR^2\times \KR^2$ of $\{0\}\times \KR^2$ (see \cite{AZ17}). Moreover, under this isomorphism the source and target maps are given as follows:
$$\bs(a,b,x,y)=(x,y)\, \y \, \,\bt(a,b,x,y)=(xe^a, bx \, \e  (a)+y).$$
where $\e\colon \KR\fto \KR$ is a smooth non-vanishing function with derivatives given by the formulas
$${\e}(a)=\left\{\begin{matrix} \frac{e^a-1}{a} & \colon \neq 0 \\
1 & \colon   a=0 \end{matrix} \right. 
\,\,\y \,\,\,
{\e}^{(n)}(a)=\left\{\begin{matrix} \left(e^a-n\,\e^{(n-1)}(a)\right)/a & \colon \neq 0 \\
1/({n+1}) & \colon   a=0 \end{matrix} \right.$$
this name is due to the Lagrange mean value theorem and the exponential function. One can verify the derivatives of $\e$ using induction and l'Hôpital rule.

One can also check that $E^*\cong TM$ and the pair groupoid $\CP(\KR^2)=\KR^2\times \KR^2$ integrates it. Using the results of the previous sections one can check that the Poisson structure on $\CP(\KR^2)$ is $\hat{\omega}=\pi \oplus (-\pi)= x'\dexp\wedge\deyp -x \dex\wedge \dey$. And the map $\varphi_\pi\colon U\fto \CP(\KR^2)$ is given by the source and target maps
$(a,b,x,y)\mapsto(xe^a,bx\, \e(a)+y,x,y)$
so
$$\begin{matrix}
   \dea\mapsto xe^a\dexp+bx\, \e'(a) \deyp & \deb\mapsto x\, \e(a) \deyp\\
   \dex\mapsto e^a \dexp+b\, \e(a) \deyp+\dex & \dey\mapsto \deyp+\dey
\end{matrix}.$$

Therefore, the only bivector in $U$ under the push-forward to $\omega$ by the map $\varphi_\pi\colon U\fto \CP(\KR^2)$ is the following one:

$$\pi_\CG= \dea\wedge\dey +\frac{1}{\e(a)}\dex\wedge\deb+\frac{b(\e(a)-1)}{a\e(a)}\deb\wedge\dey -x\dex\wedge \dey$$

This bivector is Poisson and non-degenerate, therefore it defines a symplectic structure in $U$, moreover, in the identity, when $a=b=0$, it satisfies the formula of corollary \ref{thm.sym.real.form}. $(U,\pi_\CG)$ is how the symplectic groupoid of $(\KR^2,\pi)$ looks near the identity.

\begin{rk}
 For Poisson structures on $\mathbb{R}^2$, comparable results, albeit in a less explicit form,  are available in \cite[Section 7.5]{CF}.
\end{rk}

\subsubsection{The symplectic groupoid of $(\KR^2,f(x)\dex\wedge\dey)$ for $f$ with discrete zeros}\label{sub.sub.section.f.poi}

In this subsection, we generalize the previous result on $b$-symplectic structures for a broader setting, as stated in the following theorem.

\begin{theoremC} Let $M$ be a manifold of dimension $2n$ and $(\KR^2\times M\times \KR^k,\pi)$ be a Poisson manifold with Poisson structure written as:
$$\pi(x,y,p,v)= \left(f(x,v)\dex\wedge\dey \right)\oplus \pi_0(p,v)$$
with $(x,y,p,v)\in\KR\times\KR\times M\times \KR^k$, $\pi_0(p,v)$ symplectic in $M$ for all $v\in \KR^k$ and $f(x,v)=0$ only for $(x,v)=(0,0)$. 

Any groupoid integrating the Hamiltonian foliation (the holonomy or the fundamental ones) is a regular Poisson groupoid. It looks, near the identity bisection, like an open set $\CU\subset \KR^4\times M\times M\times \KR^k$, with source and target of any $(a,b,x,y,p,q,v)\in \CU$ given by:
$$\bs\scalebox{.8}{$(a,b,x,y,p,q,v)$}=(x,y,q,v)\, \y \, \,\bt\scalebox{.8}{$(a,b,x,y,p,q,v)$}=\left(F(a,x),b G(a,x) +y,p,v\right).$$

and its multiplicative symplectic vector field will look in $\CU$ as:
 {\footnotesize
    $$\scalebox{1.2}{$\pi_\CH$}\scalebox{.8}{$(a,\!b,\!x,\!y,\!p,\!q,\!v)$}= \left(\scalebox{1.1}{$\dea\!\wedge\!\dey$} +\alpha\scalebox{.85}{$(a,\!x,\!v)$}\scalebox{1.1}{$\deb\!\wedge\!\dex$}+\scalebox{1.2}{$\frac{b}{a}$}(1\!-\!\alpha\scalebox{.85}{$(a,\!x,\!v)$})\scalebox{1.1}{$\deb\!\wedge\!\dey$} -f\scalebox{.85}{$(x,\! v)$}\scalebox{1.1}{$\dex\!\wedge\!\dey$}\right)+ \scalebox{.9}{$\pi_0(p,\!v)\! -\!\pi_0(q,\!v)$},$$}

where:
$$\alpha(a,x,v):=\left\{\begin{matrix} -\frac{f(x,v)}{G(a,x,v)} & \st a\neq 0 \y x\neq x_0\\ 
1 & \st a=0 \,\,\,\,\text{   or   } \,\,\, \,x=0\end{matrix}\right.$$ 

the function $G$ is given by the formula:
$$ G(a,x,v):=\left\{\begin{matrix}
    \frac{(x-F(a,x,v))}{a} & \st a\neq 0 \\
    -f(x,v) & \st a=0 
\end{matrix}\right.$$

and $F(a,x,v)$ is the unique function satisfying the ODE:
$$\dea F(a,x,v)=f(F(a,x,v),v) \y F(0,x,v)=x.$$

For $k=0$ this groupoid is a symplectic integration of the Poisson manifold $(\KR^2\times M\times \KR^k,\pi)=(\KR^2\times M,\pi)$.
\end{theoremC}

\begin{proof}
It suffices the case when $M=\{*\}$ and $k=0$, therefore we will only consider $\KR^2$, but the results are valid for many other Poisson manifolds, in particular, any $b^m$-Poisson since they are locally isomorphic to the Cartesian product of a symplectic manifold with the example given here.\\

In $\KR^2$ consider the Poisson bivector field $\pi=f(x) \dex\wedge \dey$ for $f\in \gi(M)$ vanishing in $x=0$ (for discrete zeros the proceeding is the same). This defines an almost regular singular foliation $E$ generated by $X:=f(x)\dex$ and $Y:=-f(x)\dey$.

For any $(a,b)\in \KR^2$ the time 1 flow of the vector field $aX+bY$ starting at $(x,y)$ is given by the formula
$$\Phi^{aX+bY}_1(x,y)=\left(F(a,x),b G(a,x)+y\right)$$
where $F(a,x)$ is the unique function satisfying the ODE:
$$\dea F(a,x)=f(F(a,x)) \y F(0,x)=x,$$
and $G$ is given by the formula:
$$ G(a,x):=\left\{\begin{matrix}
    \frac{(x-F(a,x))}{a} & \st a\neq 0 \\
    -f(x) & \st a=0 
\end{matrix}\right. .$$
The maps $\bs$ and $\bt$ from $\KR^4$ to $\KR^2$ are given by:
$$\bs(a,b,x,y)=(x,y)\, \y \, \,\bt(a,b,x,y)=\left(F(a,x),b G(a,x) +y\right).$$

One can also check that $E^*\cong TM$ and the pair groupoid $\CP(\KR^2)=\KR^2\times \KR^2$ integrates it. Using the results of the previous sections, the Poisson structure on $\CP(\KR^2)$ is $\hat{\omega}=\pi \oplus (-\pi)= f(x')\dexp\wedge\deyp -f(x) \dex\wedge \dey$; in addition, the map $\varphi_\pi\colon U\fto \CP(\KR^2)$ is given by the source and target maps
$(a,b,x,y)\mapsto\left(F(a,b),bG(x,a)+y,x,y\right)$
so for $a\neq 0$ there is:
$$\begin{matrix}
   \dea\mapsto f(F(a,x))\dexp-\frac{b\left(G(a,x)+f(F(a,x))\right)}{a} \deyp & \deb\mapsto G(a,x) \deyp\\
   \dex\mapsto (\dex F(a,x)) \dexp+\frac{b}{a} (1-\dex F(a,x)) \deyp+1 \dex & \dey\mapsto 1\deyp+1\dey
\end{matrix}.$$

Note that $H:=\dex F$ is the solution of the differential equation:
$$\dea H(a,x)=f'(F(a,x))\cdot H(a,x) \y H(0,x)=1$$
then $H=f(F(a,x))/f(x)$ .

Therefore, the only bivector in $U$ mapped by the push-forward of $\varphi_\pi\colon U\fto \CP(\KR^2)$ to $\hat{\omega}$  is the following one:

$$\pi_\CG= \dea\wedge\dey +\alpha(a,x)\deb\wedge\dex+\frac{b}{a}(1-\alpha)(a,x)\deb\wedge\dey -f(x)\dex\wedge \dey$$

where  
$$\alpha(a,x):=\left\{\begin{matrix} -\frac{f(x)}{G(a,x)} & \st a\neq 0 \y x\neq x_0\\ 
1 & \st a=0 \,\,\,\,\text{   or   } \,\,\, \,x=x_0\end{matrix}\right.$$

When $a\mapsto 0$ or $x\mapsto x_0$ then $F(a,x)\mapsto x$, $\alpha(a,x)\mapsto 1$. Indeed $\pi_\CG$ is a smooth bi-vector field in an open neighbourhood $U$ of $\{(0,0,x,y)\st (x,y)\in \KR^2\}\subset \KR^4$. 

This bi-vector is Poisson and non-degenerate, therefore it defines a symplectic structure in $U$, moreover, in the identity, when $a=b=0$, it satisfies the formula of corollary \ref{thm.sym.real.form}. $(U,\pi_\CG)$ is how the symplectic groupoid of $(\KR^2,\pi)$ looks near the identity.
\end{proof}

\begin{cor}\label{cor.bm}
    For the manifold $\KR^2$ with Poisson structure $\pi=x^m\dex\wedge\dey$, the Poisson groupoid integrating it, is locally diffeomorphic to the manifold:
    $$ \CH:=\{(a,b,x,y)\in\KR^2 \st 1-(m-1)ax>0\}\subset \KR^4,$$
    with Poisson non degenerate structure for $a\neq 0$:
    $$\pi_\CG= \dea\!\wedge\!\dey +\left(\dfrac{a\, x^{m-1}}{\scalebox{.8}{$\big(\!\scalebox{.9}{$\sqrt[m-1]{1\!-\!(m\!-\!1)ax^{m-\!1}}$}\big)^{-1}\!-\!1$}}\right)\deb\!\wedge\!\dex+\left(\frac{b}{a}\!+\!\frac{-b\, x^{m-1}}{\scalebox{.8}{$\big(\!\scalebox{.9}{$\sqrt[m-1]{1\!-\!(m\!-\!1)ax^{m-\!1}}$}\big)^{-1}\!-\!1$}}\right)\deb\!\wedge\!\dey -x^m\dex\!\wedge\!\dey\,,$$
    and for $a=0$:
    $$\pi_\CG= \dea\!\wedge\!\dey +\deb\!\wedge\!\dex -x^m\dex\!\wedge\!\dey,$$
    
    and source and target given by the following surjective submersions $\bs,\bt\colon \CH\fto \KR^2$:
     $$\bt\scalebox{.85}{$(a,\!b,\!x,\!y)$}=\left(\,x\left(\!\scalebox{.9}{$\sqrt[m-1]{1\!-\!(m\!-\!1)ax^{m-\!1}}$}\right)^{-1},\,\,\, b\,x\,\scalebox{.9}{$\dfrac{1-\big(\!\scalebox{.9}{$\sqrt[m-1]{1\!-\!(m\!-\!1)ax^{m-\!1}}$}\big)^{-1}}{a}$}+y\,\right) \y \bs\scalebox{.85}{$(a,\!b,\!x,\!y)$}=(x,y), $$
  
\end{cor}
\begin{proof}
    Using Theorem \hyperlink{thm.c}{C}  with $f(x)=x^m$ then $F(a,x)=x\big(\!\scalebox{.9}{$\sqrt[m-1]{1\!-\!(m\!-\!1)ax^{m-\!1}}$}\big)^{-1}$, $G(a,x)=x\dfrac{1-\big(\!\scalebox{.9}{$\sqrt[m-1]{1\!-\!(m\!-\!1)ax^{m-\!1}}$}\big)^{-1}}{a}$ and $\alpha(a,x)=\dfrac{-a x^{m-1}}{1-\big(\!\scalebox{.9}{$\sqrt[m-1]{1\!-\!(m\!-\!1)ax^{m-\!1}}$}\big)^{-1}}$, we get the desired formulas.
\end{proof}

In particular, for $m=2$ we get the formulas:
$$\bt(a,b,x,y)=\left(\,\frac{x}{1-ax}\,, \frac{-bx^3}{1-ax}+y\, \right) \y \bs(a,b,x,y)\mapsto(\, x\,,y\,), $$
$$\pi_\CG= \dea\wedge\dey +(1-ax)\deb\wedge\dex+bx\deb\wedge\dey -x^2\dex\wedge \dey.$$
In this case, it is easy to see that in the identity, when $a=b=0$, it satisfies the formula of corollary \ref{thm.sym.real.form}. $(\CH,\pi_\CG)$ is how the symplectic groupoid of $(\KR^2,\pi)$ looks near the identity. Moreover, the source and target maps are surjective submersions to the whole $\KR^2$ so the source and target are global symplectic resolutions for the $b^2$-structure. Under this description, only the composition (and inverse) of the symplectic groupoid is unknown.\\

For $m=3$ (and the following ones), the formulas get more complicated but, we can still simplify them as:  
$$\bt\scalebox{.8}{$(a,b,x,y)$}=\left(\,\scalebox{.9}{$\dfrac{x}{\sqrt{1-2ax^{2}}}$}\,,\, bx\,\scalebox{.9}{$\dfrac{1-\big(\scalebox{.85}{$\sqrt{1-2ax^2}$}\big)^{-1}}{a}$}\!+\!y\,\right) \y \bs\scalebox{.8}{$(a,b,x,y)$}=(\,x\,,y\,)\,, $$
$$\pi_\CG= \dea\!\wedge\!\dey +\left(\scalebox{.75}{$\dfrac{1-2ax^2+\sqrt{1-2ax^2}}{2}$}\right)\deb\!\wedge\!\dex+b\left(x^2\!+\!\scalebox{.9}{$\frac{1-\sqrt{1-2ax^2}}{2a}$}\right)\deb\!\wedge\!\dey -x^3\dex\!\wedge\!\dey\,.$$

As seen in the example above for $m\geq 3$ it gets more difficult to verify that it satisfies all the properties. In the case of $m=3$ one can check smoothness and the correct limits for $a\mapsto 0$ by applying l'Hôpital's rule to the function $\scalebox{.8}{$\big(1-\big(\scalebox{.8}{$\sqrt{1-2ax^2}$}\big)^{-1}\big)$}/{a}$.
\subsubsection{The integrating groupoid for a desingularization of $b^{2k}$-symplectic manifolds}\label{sec:b.desing}

Can a singular symplectic structure be desingularized?
Recall from \cite{GMW19} the following result.
\begin{thm}[Guillemin-Miranda-Weitsman]

Given a $b^m$-symplectic structure $\omega$ on a compact manifold $(M^{2n},Z)$:
\begin{itemize}
\item If {$m=2k$}, there exists  a family of{ symplectic forms ${\omega_{\epsilon}}$} which coincide with  the $b^{m}$-symplectic form
    $\omega$ outside an $\epsilon$-neighbourhood of $Z$ and for which  the family of bivector fields $(\omega_{\epsilon})^{-1}$ { converges } in
    the $C^{2k-1}$-topology to the Poisson structure $\omega^{-1}$ as $\epsilon\to 0$ .
\item If {$m=2k+1$}, there exists  a family of { folded symplectic forms ${\omega_{\epsilon}}$} which coincide with  the $b^{m}$-symplectic form
    $\omega$ outside an $\epsilon$-neighbourhood of $Z$.

\end{itemize}

\end{thm} This desingularization process will be called \emph{GMW desingularization}.
The GMW desingularization  sheds light on the topological obstructions that a given manifold has in order to admit a $b^m$-symplectic structure:
\begin{itemize}
  
  \item Any $b^{2k}$-symplectic manifold admits a symplectic structure.
  \item Any $b^{2k+1}$-symplectic manifold admits a folded symplectic structure.
  \item The converse is not true: as {$S^4$} admits a folded symplectic structure but no $b$-symplectic structure \cite{Cannas}.
\end{itemize}

For the even case, $b^{2k}$ , the GMW-desingularization process assigns a family of symplectic structures. The idea of the proof is as follows

Write the $b^{2k}$-symplectic form as:
{\begin{equation}\label{00} \omega =\, \frac{dx}{x^{2k}}\!\wedge\left(\sum_{i = 0}^{2k-1}  \alpha_{i}x^i\right) \,+\, \beta \end{equation}}

 \begin{itemize}
\item  $h\in \mathcal{C}^{\infty}(\mathbb R)$ odd function s.t. $h'(x)>0$ for $x \in [-1,1]$, and such that outside $[-1,1]$,
\[h(x)=\begin{cases} \frac{-1}{(2k-1) x^{2k-1}}-2 &\textrm{for} \quad x<-1  \\[.5em] \frac{-1}{(2k-1) x^{2k-1}}+2 &\textrm{for} \quad x>1
\end{cases}.\]
under these assumptions, the function $h$ is injective, which is important because its inverse will be used later on.

\item Re-scale $h_\epsilon(x)=\frac{1}{\epsilon^{4k-2}}h(\frac{x}{\epsilon^2})$ on $\epsilon\in (-1,1)\backslash \{0\}$. This function does not converge for $\epsilon\mapsto 0$ but its derivative with respect to $x$ does.

The picture below depicts in red the graph\footnote{The red curve is a graph for $\frac{4}{\pi}\arctan(x)$, which has similarities with $h(x)$: smooth, same values in x=-1 and in x=1, same limits at infinity and with positive derivative for  $x\in(-1,1)$.} of the function $h$ . Its rescaled versions are shown in blue and green.

\begin{center}
\includegraphics[scale=0.13]{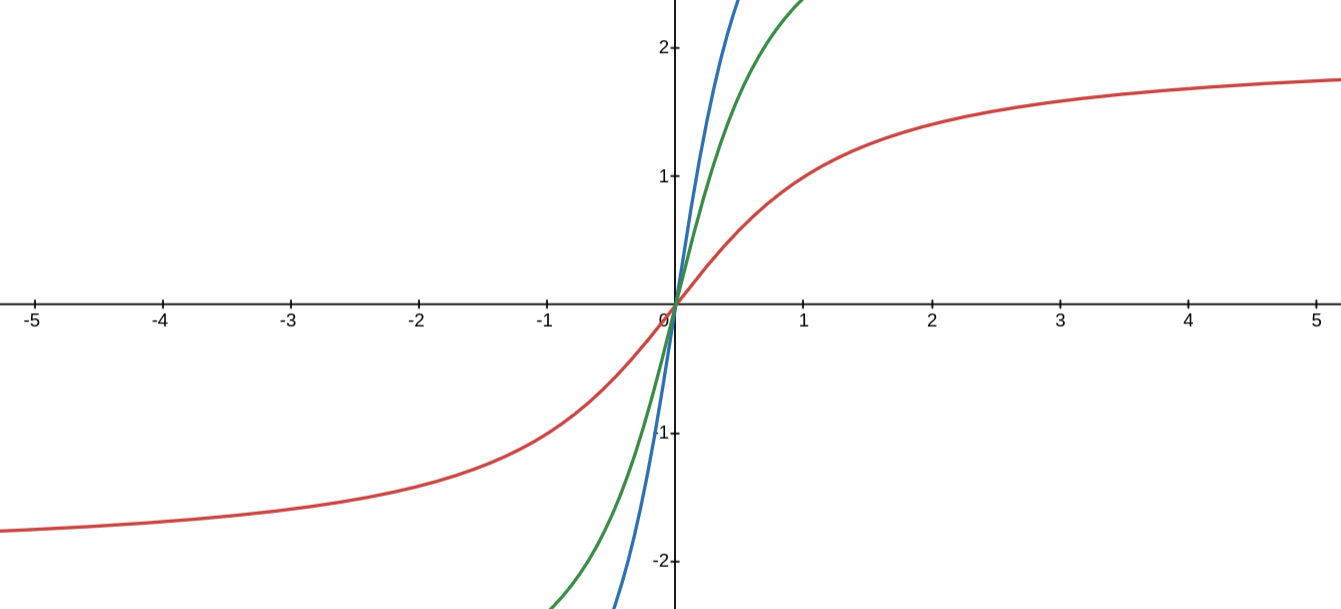}
\end{center}
\item Replace $\frac{dx}{x^{2k}}$ by $h_\epsilon' dx$  to obtain{${\omega_\epsilon} = h_\epsilon' dx\wedge(\sum_{i = 0}^{2k-1} \alpha_{i}x^i) + \beta$} which is symplectic for $\epsilon\neq 0$ and converges to the original $b^{2k}$-symplectic structure when $\epsilon\mapsto 0$ in the $C^{2k-1}$ topology.
\end{itemize}

Using the fact that \( h_\epsilon'(x) > 0 \), we define the function
\[
g_\epsilon(x) = \frac{1}{f_\epsilon'(x)} - x^{2k},
\]
for \(\epsilon \in (-1,1)\). This function satisfies the following properties:
\begin{itemize}
    \item \(g_\epsilon(x) \geq 0\) for \(x \in (-\epsilon^2, \epsilon^2)\),
    \item \(g_\epsilon(0) > 0\) when \(\epsilon \neq 0\),
    \item \(g_\epsilon(x) = 0\) for \(x \notin (-\epsilon^2, \epsilon^2)\),
    \item \(g_\epsilon(x)\) tends to zero in the \(C^{2k-1}\) topology as \(\epsilon \to 0\).
\end{itemize}

Moreover, we have
\[
dh_\epsilon(x) = \frac{dx}{x^{2k} + g_\epsilon(x)}.
\]

Below is an example graph illustrating a possible shape of \(g_\epsilon\):
\begin{center}
\includegraphics[scale=0.2]{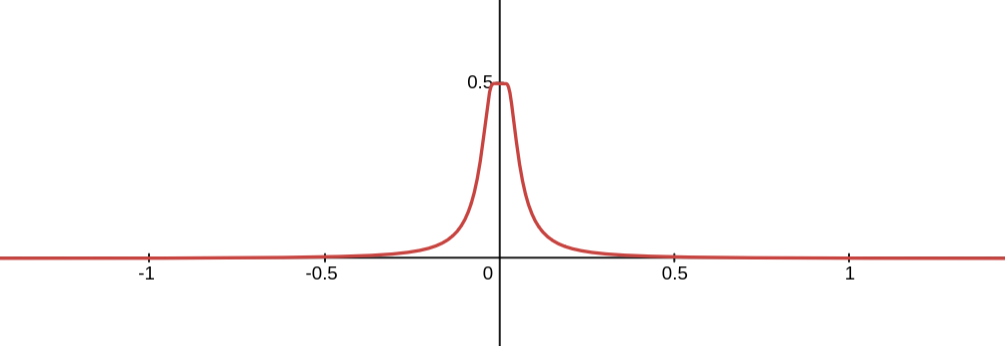}
\end{center}

The form $\omega_\epsilon$ can be written locally as

$${\omega_\epsilon} = \frac{dx}{x^{2k}+g_\epsilon (x)}\!\wedge\!\left(\sum_{i = 0}^{2k-1}  \alpha_{i}x^i\right) + \beta$$
And the Poisson bivector counterpart:
$$\pi_\epsilon = \left(\,x^{2k}+g_\epsilon(x)\,\right) \left(\dex_{\scriptscriptstyle\!1}\! \wedge\! \dey_{\scriptscriptstyle\!1}\right)\, +\, \pi_\beta$$
We now prove the following theorem that connects the integrating structure described above with the desingularization.

\begin{thm} Let $(M, \omega_\epsilon)$ be the desingularization of a $b^{2k}$-symplectic manifold. Denote by $\pi_\epsilon$ the dual Poisson structure. Then:

\begin{enumerate}
    \item The pair $\left(M \times ( 1, -1), \pi_{\epsilon}\right)$ is an almost regular Poisson manifold. 
    \item Let $E=\pi_{\epsilon}^\sharp\Omega(M \times ( 1, -1))$ be the Hamiltonian foliation and $E^*$ its associated dual, which is a regular foliation. The leaves of $E^*$ are the subsets $M\times\{\epsilon\}\subset M\times(-1,1)$ with $\epsilon$ constant. The holonomy groupoid integrating $E^*$ coincides with the pair groupoid $\mathcal{M}=M\times M\times (-1,1)$ and Poisson structure $\pi_\epsilon\oplus -\pi_\epsilon\oplus 0$.

    \item The (full) holonomy groupoid integrating $E$, $(\CH,\tilde{\pi_\epsilon})$ is a blow-up of $(\mathcal{M},\hat{\pi})$ near $\epsilon=0$, in the sense that:
    The map $\bt\times\bs\colon \CH\fto \left(M\times (-1,1)\right)\times \left(M\times(-1,1)\right)$ lies in the diagonal of $(-1,1)$, therefore by abuse of notation, it is a map $\bt\times\bs\colon \CH\fto M \times M\times(-1,1)=\mathcal{M}$, this map is Poisson diffeomorphism everywhere except in $\epsilon=0$.

    \item Near any point in $M\times(-1,1)$ with $x=0$ and $\epsilon=0$ there is a neighbourhood $U\subset M\times(-1,1)$, such that the holonomy groupoid near the identity bisection in $U$ is diffeomorphic to an open neighbourhood of $\KR^{2n}\times U\subset \KR^{2n}\times  M\times(-1,1)$, and with Poisson structure  written as:
    {\footnotesize 
    $$\scalebox{1.2}{$(\pi_\CH)|_{\CU}$}= \left(\scalebox{1.2}{$\dea\!\wedge\!\dey$} +\alpha\scalebox{.85}{$(a,\!x,\!\epsilon)$}\scalebox{1.2}{$\deb\!\wedge\!\dex$}+\scalebox{1.2}{$\frac{b}{a}$}(1\!-\!\alpha\scalebox{.85}{$(a,\!x,\!\epsilon)$})\scalebox{1.2}{$\deb\!\wedge\!\dey$} -(x^{2k}\!+\!g_\epsilon\scalebox{.85}{$(x)$})\scalebox{1.2}{$\dex\!\wedge\!\dey$}\right)+ \scalebox{.9}{$\pi_\beta(p)\! -\!\pi_\beta(q)$},$$}

where $(a,b)\in \KR^2,p\in \KR^{2n-2},(x,y,q)\in M,\epsilon\in (-1,1)$, $(a,b,p,x,y,q,\epsilon)\in \KR^{2n}\times U\subset \KR^{2n}\times  M\times(-1,1)$, and:
  \[
    \alpha(a,x,\epsilon):=\begin{cases} 
    \hspace{.6in} 1\,\,, & \text{if } a=0 \text{ or } x=0, \\[.8em]
    \,\,\scalebox{.9}{$\dfrac{a\, x^{2k-1}}{\left(\scalebox{.85}{$\sqrt[2k-1]{\scalebox{.85}{$1-(2k-1)a\,x^{2k-1}$}}$}\right)^{-1}-1}$}\,\,, & \text{if } a\neq 0,\ x\neq 0,\ \epsilon= 0,\\[1.3em]
    
   \hspace{.2in} \dfrac{a\,\bigl(\,x^{2k}+g_\epsilon(x)\,\bigr)}{h_\epsilon^{-1}\!\scalebox{.85}{$\bigl(a\!+\!h_\epsilon(x)\bigr)$}-x}\,\,, & \text{if } a\neq 0,\ x\neq 0,\ \epsilon\neq 0.
    \end{cases}
    \]

\end{enumerate}
\end{thm}

\begin{proof}
We can prove the first three items with the following argument. The Hamiltonian foliation is generated by the set of vector fields:
\[
\left\{(x^{2k}+g_\epsilon(x))\,\dex,\quad(x^{2k}+g_\epsilon(x))\,\dey,\quad \pi_\beta^\sharp(\gw(M))\right\},
\]
and is therefore an almost regular foliation in \(M\times(-1,1)\).\\

Moreover, by Theorem \hyperlink{thm.c}{C}, taking \(v = \epsilon\) and \(f(x,\epsilon) = x^{2k} + g_\epsilon(x)\), we obtain the explicit formulas for the Poisson bivector on the Lie groupoid (last item).

The formula for \(\alpha(a,x,\epsilon)\) in Theorem \hyperlink{thm.c}{C} involves a function \(F(a,x,\epsilon)\) as the unique function satisfying the ODE:
\[
\frac{\partial}{\partial a} F(a,x,\epsilon)
= \bigl(F(a,x,\epsilon)\bigr)^{2k} + g_\epsilon\bigl(F(a,x,\epsilon)\bigr)
= \frac{1}{h'_\epsilon\bigl(F(a,x,\epsilon)\bigr)},
\quad
F(0,x,\epsilon) = x.
\]

This implies that:
\[
F(a,x,\epsilon) := 
\begin{cases}
\dfrac{x}{\sqrt[2k-1]{1 - (2k-1)a x^{2k-1}}}, & \text{if } \epsilon=0,\\[1.2em]
h_\epsilon^{-1}\bigl(a + h_\epsilon(x)\bigr), & \text{if } \epsilon \neq 0,
\end{cases}
\]
and then we substitute \(F\) into the formula for \(\alpha\).
\end{proof}

\begin{rk}
Observe that the smoothness of \( F \) follows from the existence and uniqueness of the solution to the ODE, which is guaranteed under appropriate regularity conditions on the function \( h_\epsilon \).
\end{rk}

\subsection{On the symplectic integration of almost regular Poisson manifolds}\label{sec.ar.poi}

Let $(M,\pi)$ be an almost regular Poisson manifold, with almost regular foliation $E=\pi^\sharp(\gw(M))$ and ai-algebroid $\EA$ with anchor $\rho$. Let $\lambda$ the vector bundle morphism such that $\pi^\sharp=\rho\circ\lambda$, so there is a short exact sequence of Lie algebroids:

$$\ker(\lambda)\fto T^* M \xrightarrow{\lambda} \EA,$$

where the far left and the far right elements are integrable. The left side by a bundle of Lie groups and the right-hand side by the Poisson (holonomy) groupoid of $E$. 

Giving any splitting $\alpha: T^* M\fto \ker(\lambda)$ of the short exact sequence, there is an isomorphism of vector bundles
$$T^*M\fto \ker(\lambda)\oplus \EA, \theta\mapsto \alpha(\theta)\oplus \lambda(\theta);$$
This isomorphism gives a Lie bracket on $\gs(\ker(\lambda)\oplus \EA)$ and a Lie algebroid structure. 

If one can find a splitting of the short exact sequence such that the Lie algebroid structure in $\ker(\lambda)\oplus \EA$ is the piecewise algebroid structure, then the symplectic groupoid integrating $T^*M$ can be expressed using the source-simply connected groupoids integrating $\EA$ and $\ker(\lambda)$.

\begin{cor}
    If $\EA$ is of full rank then $\CH(\EA)$ has a symplectic structure integrating the Poisson structure $T^*M$.
\end{cor}

The corollary above is a consequence that for full rank the map $\lambda$ is an isomorphism so $\ker(\lambda)=0$. One can refer to Subsections \ref{sub.sub.section.b.sym} and \ref{sub.sub.section.f.poi} for the cases of $b$-symplectic structures, $b^2$-symplectic, and more generally $b^m$-symplectic structures, as well as $f(x) \dex \wedge \dey$ in $\mathbb{R}^2$. The groupoids presented there are the ones integrating their corresponding Poisson manifolds. In the case of $b$-symplectic manifolds, this provides a new description of the symplectic groupoid that integrates the underlying Poisson structure, which is also studied in \cite{gualtieri} and \cite{GMP14}.

\subsubsection{The case of cosymplectic manifolds}\label{sub.sub.section.cosymp}
In this section, we can apply the technique explained in the previous section. For cosymplectic manifolds, the splitting of $\lambda$ always exists. Let us first review the definition of cosymplectic manifolds.

\begin{defn}
    A cosymplectic manifold is a triple $(M,\omega,\alpha)$ where $M$ is a manifold of odd dimension $2n+1$, $\omega$ and $\alpha$ are forms $\omega\in \gw^2(M)$ and $\alpha\in\gw(M)$ and such that $\omega^n\wedge \alpha$ is nowhere zero (it is a volume form).
\end{defn}

Every cosymplectic manifold $(M,\omega,\alpha)$ induces an almost regular Poisson structure and an $E$-symplectic structure on $M$. Indeed, let $E$ be the induced foliation by the Lie algebroid $\EA=\ker(\alpha)$. As a consequence of $\omega^n\wedge \alpha$ being a volume form we get that $\omega$ restricted to $\EA$  is an $E$-symplectic structure. Moreover, the Poisson structure is given as follows: if $\iota\colon \EA\fto TM$ is the inclusion, and $\omega^\sharp \colon \EA^*\fto \EA$ is the map induced by $\omega$, then the Poisson manifold is given by 
$$\pi^\# :=\iota\circ\omega^\sharp\circ \iota^* \colon T^*M\fto TM.$$

Before we state the main result of this section let us recall the following definitions.

\begin{defn}
    Given $\CG\soutar M$ a Lie groupoid, a form $\hat{\beta}\in \gw(\CG)$ is multiplicative if and only if:
    \[m^*\hat{\beta}={\rm pr}_1^*\hat{\beta}+{\rm pr}_2^*\hat{\beta},\]
    where $m\colon \CG \times_M \CG \fto \CG$ is the multiplication and ${\rm pr}_i\colon \CG \times_M \CG \fto \CG$ is the projection on the $i$'th coordinate.
\end{defn}

\begin{defn}
    A cosymplectic groupoid is a Lie groupoid $\CG\soutar M$ with a multiplicative cosymlectic structure $(\hat{\omega},\hat{\alpha})$. 
\end{defn}

The main result of this section is as follows:

\begin{thm}

Given a cosymplectic manifold $(M,\omega,\alpha)$ its induced Poisson structure $\pi^\sharp:T^* M\fto TM$ is integrable. 
    
More precisely: for $\EA=\ker(\alpha)$
\begin{enumerate}
\item The forms $\hat{\omega}:=\bt^* \omega-\bs^*\omega\in \gw^2(\CH(\EA))$ and $\hat{\alpha}:=\bt^*\alpha=\bs^*\alpha\in \gw^1(\CH(\EA))$ are well defined and multiplicative.
    \item $(\CH(\EA),\hat{\omega},\hat{\alpha})$ is a cosymplectic groupoid.
    \item the Lie algebroid $T^*M$ is isomorphic to the Lie algebroid $\KR\times \EA$ with bracket bracket given by:
$$[(f_1,v_1),(f_2,v_2)]=(\CL_{v_1}f_1- \CL_{v_2} f_2 , [v_1,v_2]).$$
\item The Lie groupoid integrating $T^*M$ is isomorphic to the symplectization of $\CH(\EA)$, i.e., it is isomorphic to $(\KR\times \CH(\EA),\tilde\omega)$ where $\tilde\omega= \left(d(p\cdot \hat{\alpha}) + \hat{\omega}\right)$ and $p\colon \KR\times \CH(\EA)\fto \KR$ is the projection into the first component.
\end{enumerate}
\end{thm}

\begin{proof}

Part 4 of this theorem is a direct consequence of part 3. Therefore, we will only prove parts 1, 2, and 3.

For part 1, recall that \(\EA=\ker(\alpha)\); therefore,
\[
\ker(\bt_*)\cup \ker(\bs_*)\subset \ker(\bt^*\alpha)\cap \ker(\bs^*\alpha).
\]
Moreover, for any bisection \(\sigma\) of \(\CH(\EA)\) and any \(v\in T\CH(\EA)\), one can write
\[
v=k+w,
\]
where \(k\in \ker(\bs_*)\) and \(w\in T\sigma\). The fact that \(\sigma\) is a bisection of \(\CH(\EA)\) implies that it induces a diffeomorphism of \(M\) that preserves \(\alpha\); hence, \(\alpha(\bs_*w)=\alpha(\bt_*w)\). Consequently, we have
\[
\bs^*\alpha(k+w)=\bs^*\alpha(w)=\bt^*\alpha(w)=\bt^*\alpha(k+w),
\]
which shows that \(\bs^*\alpha=\bt^*\alpha\). The forms \(\hat{\omega}\) and \(\hat{\alpha}\) are multiplicative by definition.

For part 2, if \(v\in \ker(\hat{\omega}^\flat)|_M\) then 
\[
\bt_*v-\bs_*v\in \ker(\alpha)
\]
and 
\[
\hat{\omega}(v,-)=\omega(\bt_*v-\bs_*v,-)=0.
\]
Since \(\omega\) is symplectic on \(\ker(\alpha)\), it follows that \(\bt_*v=\bs_*v\). Moreover, if we also have \(\hat{\alpha}(v)=0\), then \(\bt_*v=\bs_*v=0\) and consequently \(v=0\) (because \(\CH(\EA)\) is the groupoid of a regular foliation and has discrete isotropy groups). Since
\[
\ker(\hat{\omega}^\flat)\cap\ker(\hat{\alpha})=0,
\]
and both \(\hat{\omega}\) and \(\hat{\alpha}\) are closed, we deduce that \((\hat{\omega},\hat{\alpha})\) defines a cosymplectic structure on \(\CH(\EA)\). The remainder of the statement follows from multiplicativity.

For part 3, observe that the kernel \(\ker(\pi^\#)=\langle\alpha\rangle\) is a trivial line subbundle of \(T^*M\). Let us denote
\[
l_\alpha:=\ker(\pi^\#).
\]
Moreover, the image of \(\pi^\#\) is \(\EA\). This yields the following short exact sequence of Lie algebroids:
\[
l_\alpha \fto T^*M \fto \EA.
\]

For any cosymplectic manifold, there is a unique nonvanishing vector field \(K\in \mathfrak{X}(M)\), called the Reeb vector field, given by the formulas
\[
\alpha(K)=1 \quad \text{and} \quad \iota_K\omega=0.
\]
The following map preserves the Lie brackets:
\[
D_K\colon T^*M \fto l_\alpha,\quad \beta\mapsto \beta(K)\alpha,
\]
as a consequence of \(d\alpha=0\) and \(\pi^\#(\alpha)=0\).

This gives us a diffeomorphism
\[
T^*M \fto l_\alpha\oplus \EA \cong \KR\times \EA,\quad \beta\mapsto \bigl(D_K \beta,\,\pi^\#\beta\bigr).
\]

\end{proof}

When \(M\) and one of the leaves \(L_0\) of \(E\) are compact manifolds, this result agrees with the integration given by \cite[Corollary 26]{GMP11}. Indeed, in this case \(M\) is diffeomorphic to a mapping torus \(S^1\times_f L_0\), where 
\[
f\colon L_0\fto L_0
\]
is a diffeomorphism preserving the symplectic form \(\omega|_{L_0}\). Under this diffeomorphism the one-form \(\alpha\) coincides with \(dq\), where 
\[
q\colon S^1\times_f L_0\fto S^1
\]
is the projection onto the first component, and \(\omega\) is the pull-back of \(\omega|_{L_0}\) to \(S^1\times_f L_0\), using the fact that \(f\) is a symplectomorphism. 

In this case, the cosymplectic structure on the holonomy groupoid 
\[
\CH(\EA)\cong S^1\times_f (L_0\times L_0)
\]
is given by \(\alpha=dq\) and \(\hat\omega\) is obtained by lifting the symplectic form
\[
{\rm Pr}_1^*\omega|_{L_0}-{\rm Pr}_2^*\omega|_{L_0}
\]
to 
\[
S^1\times_{(f\times f)} (L_0\times L_0),
\]
using the symplectomorphism \(f\times f\).

Therefore, the symplectization of \(\CH(\EA)\), and hence the groupoid integrating \(T^*M\), is symplectomorphic to:
\[
\Bigl(\KR\times \CH(\EA),\, \tilde\omega\Bigr)
\cong \Bigl(\,T^*S^1\!\times_{(f\times f)}\! (L_0\!\times\! L_0)\,\,,\, d\theta\!+_{(f\!\times\! f)}\!\Bigl({\rm Pr}_1^*\omega|_{L_0}\!-\!{\rm Pr}_2^*\omega|_{L_0}\Bigr)\,\Bigr)\,,
\]
where $\theta$ is the canonical Liouville form on $T^*S^1$.

\subsection{The Poisson groupoid integrating $E$ for $E$-symplectic manifolds}\label{sec.int.e}
If $(M,A, \omega)$ is an $E$-symplectic manifold, there is no guarantee of the existence of a commutative frame for either $A$ or $A^*$. Nevertheless, the map $\omega^\sharp\colon A\fto A^*$ is a Lie algebroid isomorphism, meaning that $\CG\cong \CG^*$ i.e., there is a self T-duality.

We can also verify that the Poisson structure in $\CG$ comes from an $s^{-1}(E)$-symplectic structure.  By \cite{GZ19}, the foliations $E$ and $\bs^{-1}E$ are Morita equivalent and have the same type of singularities; this implies that the singularities of the Poisson structure in $\CG$ resemble those in $M$. Further details are provided in the following Theorem.

\begin{theoremD}\label{thm.e.symp.grpd} 

Let $(M,A, \omega)$ be an $E$-symplectic manifold. The pullback Lie algebroid $A_\CG=\bs^{!}A=\bs^*A {}_\rho \!\times_{d\bt} T\CG$ (do not confuse with the pullback vector bundle $\bs^* A$) is almost injective. Moreover, there is a symplectic structure $\omega_\CG$ in $A_\CG$ such that the triple $$\left(\CG, A_\CG, \omega_\CG\right)$$ is an $\bs^{-1}(E)$-symplectic manifold, with the Poisson structure being the canonical one induced by the Lie bialgebroid $(A,A^*)$.
\end{theoremD}

\begin{proof}
By the results of Section 1.2 of \cite{GZ19}, if $E$ is the singular foliation of the Lie algebroid $A$, then $\bs^{-1}(E)$ is the singular foliation of the Lie algebroid $\bs^! A$. Moreover, if $A$ is an ai-algebroid, then $\bs^!A$ is also an ai-algebroid. This implies that $\bs^{-1}E$ is an almost regular singular foliation with Lie algebroid $A_\CG$.\\

 As in  theorem \ref{thm.e.symp.grpd0}, the vector bundle $A_\CG'=\bt^* A \oplus \bs^* A$ has a bivector field $\pi_\CG'=\pi_\omega\oplus -\pi_\omega$ and a diffeomorphism $\varphi:A_\CG'\fto \bs^!A$ defining the Poisson structure $\pi_\CG$ in $\CG$. The bivector $\varphi^2\pi_\CG'$ is clearly non-degenerate and Poisson in $A_\CG$. The dual $\omega_\CG$ is symplectic (non-degenerate and closed), making the triple mentioned in the proposition an $\bs^{-1}(E)$-symplectic manifold.
\end{proof}
There is a well known case, which is for any $b^m$-Poisson manifold, the groupoid integrating the $b^m$-foliation is also a $b^m$-Poisson manifold.
For the rest of this section, we are going to see ways to write the E-symplectic form locally.
\begin{prop}
  For any point, there exists a local frame $\alpha_1, \beta_1, \dots, \alpha_k, \beta_k$ for ${}^E\!A^*$ such that the symplectic structure can be expressed locally as
$$\omega=\sum_i \alpha_i\wedge \beta_i.$$
\end{prop}

\begin{proof}
The fact that $({}^E\!A, \omega)$ is a symplectic vector bundle implies that this proposition is true. One starts with a non-vanishing vector field $X_1 \in \gx(U)$, and then, using the Gram-Schmidt process described in Section 1.1 of \cite{AC08}, it is possible to find $Y_1, X_2, \dots, X_k, Y_k$ such that
$\omega(X_i,Y_j)=\delta_{ij}$ $ \omega(X_i,X_j)=\omega(Y_i,Y_j)=0$. Then $\alpha_i=\omega(Y_i,-)$ and $\beta_i=\omega(X_i,-)$.

 \end{proof}

\begin{prop}\label{prop.darboux.e} Let $\omega$ be an $E$-symplectic structure. Assume that the form $\omega$ decomposes as $\sum_i \alpha_i \wedge \beta_i$. Then $\alpha_i$ and $\beta_i$ are closed if and only if $\alpha_i, \beta_i$ form a commutative frame for $\EA^*$.
\end{prop}
\begin{proof}
We know that $\omega$ is closed; therefore,
    $$0\!=\!d\omega=[\omega,\omega]=\!\sum_{ij}\left([\alpha_j,\alpha_i]\!\wedge\! \beta_j\!\wedge\!\beta_i -\alpha_j\!\wedge\![\alpha_i,\beta_j]\!\wedge\!\beta_i +\alpha_i\!\wedge\![\alpha_j,\beta_i]\!\wedge\!\beta_j -\alpha_i\!\wedge\!\alpha_j\!\wedge\![\beta_i,\beta_j]\right).$$
    This implies:
    $$\begin{array}{ccc}
       [\alpha_j,\alpha_i]\in {\rm Span}_{\gi(M)}(\beta_j,\beta_i) \,\,,  & [\alpha_j,\beta_i]\in {\rm Span}_{\gi(M)}(\beta_j,\alpha_i) \,\, , & [\beta_j,\beta_i]\in {\rm Span}_{\gi(M)}(\alpha_j,\alpha_i) \\
    \end{array}.$$

    Moreover, an element $\kappa\in \gs(\EA^*)$ is closed if and only if
    $$0=d\kappa=[\omega,\kappa]=\sum_i[\alpha_i,\kappa]\wedge \beta_i +\alpha_i\wedge [\beta_i,\kappa].$$
 This means:
     $$\begin{array}{cc}
       [\alpha_i,\kappa]\in {\rm Span}_{\gi(M)}(\beta_i) \, \, , & [\beta_i,\kappa]\in {\rm Span}_{\gi(M)}(\alpha_i) \\
    \end{array}.$$
 Thus, it is clear that $\alpha_i$ and $\beta_i$ are closed if and only if they commute with any other $\alpha_j$ and $\beta_j$.
\end{proof}
This clearly obstructs expressing $\omega = \sum_i \alpha_i \wedge \beta_i$ with $\alpha_i$ and $\beta_i$ closed, as in the Darboux-type normal form.

\begin{cor}
    In particular, if $\EA$ is the zero tangent bundle, no $\EA$-symplectic structure can take the form described in Proposition \ref{prop.darboux.e} with $\alpha_i$ and $\beta_i$ closed.
\end{cor}
Observe that there are two main ingredients in Proposition \ref{prop.darboux.e}: The first is the assumption of the existence of a \emph{splitted} form decomposition as $\omega = \sum_i \alpha_i \wedge \beta_i$. The second condition is that $\alpha_i$ and $\beta_i$ are closed 1-forms in the $E$-complex. The proposition is stated in terms of the existence of a splitted form but how does one obtain such a splitting?. For certain $E$-structures, such splitted forms are intrinsically related to Darboux normal forms.

These results, which relate the existence of commutative frames to the splitting of Darboux-type normal forms and closed forms, mirror the classical Darboux-Carathéodory theorems under the assumption of the existence of a set of Poisson-commuting functions. The Darboux-Carathéodory normal form is of total type when the number of Poisson-commuting functions is maximal, naturally leading us to the realm of integrable systems.

In the next subsection, we analyze in detail two Darboux-Carathéodory theorems under the assumption of the existence of first integrals for (regular) Poisson manifolds and $b$-symplectic manifolds.

\subsubsection{Darboux-Carathéodory theorem and commutative frames}

For integrable systems on regular Poisson manifolds and $b$-symplectic manifolds (or more generally $b^m$-symplectic manifolds), these Darboux-Carathéodory type theorems have been established in the literature. The Darboux-Carathéodory theorems are a key step in proving action-angle coordinate theorems, which can be seen as cotangent models \cite{KM17}.

Recall that both regular Poisson manifolds and $b$-symplectic, or more generally $b^m$-symplectic, manifolds are examples of $E$-symplectic manifolds.

We present the statements of such theorems and connect them to the earlier proposition. For simplicity of exposition, we focus on the case of $b$-manifolds. A Darboux-Carathéodory theorem for $b^m$-symplectic manifolds is obtained in \cite{MP}.

\subsubsubsection{The case of regular Poisson manifolds}

The following theorem, proved in \cite{LaurentMirandaVanhaecke}is a generalization of the Carathéo\-dory-Jacobi-Lie theorem \cite[Th.\
 13.4.1]{libermannmarle} for an arbitrary Poisson manifold $(M,\Pi)$. It provides a set of canonical local
coordinates for the Poisson structure $\Pi$, which contains a given set $p_1,\dots,p_r$ of functions in which pairwise commute for the Poisson bracket. This result is often used for the proof of the action-angle-type theorems.

\begin{thm}[Laurent-Miranda-Vanhaecke, \cite{LaurentMirandaVanhaecke}] \label{thm:localsplitting}
  Let $m$ be a point of a Poisson manifold $(M,\Pi)$ of dimension~$n$. Let $p_1,\dots,p_r$ be $r$ functions in
  involution, defined on a neighbourhood of $m$, which vanish at $m$ and whose Hamiltonian vector fields are
  linearly independent at $m$.  There exist, on a neighbourhood $U$ of $m$, functions
  $q_1,\dots,q_r,z_1,\dots,z_{n-2r}$, such that
  \begin{enumerate}
    \item The $n$ functions $(p_1,q_1,\dots,p_r,q_r, z_1,\dots,z_{n-2r})$ form a system of coordinates on $U$,
    centered at $m$;
    \item The Poisson structure $\Pi$ is given on $U$ by
    \begin{equation}\label{eq:thm_split}
      \Pi=\sum_{i=1}^r\pp{}{q_i}\we\pp{}{p_i}+\sum_{i,j=1}^{n-2r} g_{ij}(z)\pp{}{z_i}\we\pp{}{z_j},
    \end{equation}
    where each function $g_{ij}(z)$ is a smooth function on $U$ and is independent of $p_1,\dots,p_r,
    q_1,\dots,q_r$.
  \end{enumerate}
  The rank of $\Pi$ at $m$ is $2r$ if and only if all the functions $g_{ij}(z)$ vanish for $z=0$.
  Such local coordinates $(p_1,q_1,\dots,p_r,q_r, z_1,\dots,z_{n-2r}) $ are called \emph{coordinates conjugated to
  $p_1,\dots,p_r$}.
\end{thm}

Observe that Proposition \ref{prop:comm.chart.grpd} is given in terms of split forms. Whenever we consider a regular foliation we can consider the $E$-forms as foliated forms. In particular, a regular Poisson structure defines a closed $2$-form in the $E$-complex of foliated forms (see also, \cite{MS21}). 

The theorem above applied to the case of regular Poisson manifolds yields the desired splitted form.

As we want the forms to be split, the case we are more interested in would be the following Corollary which holds for regular Poisson manifolds.
\begin{cor} 
    When the rank of $\Pi$  is $2r$  the Poisson structure can be written as:

  \begin{equation}\label{eq:thm.split.poi}
      \Pi=\sum_{i=1}^r\pp{}{q_i}\we\pp{}{p_i},
    \end{equation}

    and the dual form in the $E$-complex can be written as the split $2$-form:

      \begin{equation}\label{eq:thm.split.sym}
      \omega_\Pi=\sum_{i=1}^r d{q_i}\we d{p_i},
    \end{equation}
    
  Such local coordinates $(p_1,q_1,\dots,p_r,q_r, z_1,\dots,z_{n-2r}) $ are called \emph{coordinates conjugated to
  $p_1,\dots,p_r$}.
\end{cor}

The proof of this theorem builds upon the existence of a commutative frame and proceeds by employing a Frobenius-type argument. The existence of commuting functions allows for a more precise choice of commutative frames, leading to the derivation of exact one-forms (analogous to the $\beta_i$ in Proposition \ref{prop.darboux.e}). In essence, the proof uses the joint flow of the distributions generated by the Hamiltonian vector fields of the first integrals. Relating this to the constructions in \cite{CF03} would be interesting.

\subsubsubsection{The case of $b$-symplectic manifolds}
For $b$-symplectic manifolds, a Darboux-Carathéodory theorem was proven in \cite{KMS16}. The notion of commuting functions is tricky in this case as one needs to extend the class of smooth functions to include log-terms.
The Darboux-Carathéodory theorem for $b$-manifolds is proved in \cite{KMS16}. Important applications of this theorem in the non-commutative set-up are discussed in \cite{KM16}. 

Let us begin with the definition of a $b$-integrable system. For that we need to consider the set of $b$-functions: $^b C^\infty(M)$, which consists of functions with values in $\mathbb{R} \cup \{\infty\}$ of the form
$$c\,\textrm{log}|f| + g,$$
 where $c \in \mathbb{R}$, $f$ is a defining function for the critical set $Z$ of the $b$-manifold $(M,Z)$, and $g$ is a smooth function.

\begin{defn} [{\bf $b$-integrable system}]\label{def
}
A $b$-integrable system on a $2n$-dimensional $b$-symplectic manifold $(M^{2n},\omega)$ is a collection of $b$-functions $F=(f_1,\ldots,f_n)$ that satisfy the following conditions: \begin{itemize} \item The functions pairwise commute under the Poisson bracket: $\{f_i,f_j\} = 0$ for all $i, j$; \item $df_1 \wedge \dots \wedge df_n$ is nonzero as a section of $\Lambda^n({^b}T^*(M))$ on a dense subset of $M$ and on a dense subset of $Z$. \end{itemize} Points in $M$ where the second condition holds are called {\bf regular} points. \end{defn}

\begin{rk} Note that $df_1 \wedge \dots \wedge df_n$ is nonzero at a point in $\Lambda^n({^b}T^*M)$ if and only if the vector fields $X_{f_1}, \dots, X_{f_n}$ are linearly independent at that point. This follows from the fact that the map ${^b} TM \to {^b} T^*M$, induced by $\omega$, is bijective. Moreover, this condition implies that at least one of the $f_i$ must be a genuine $b$-function, i.e., non-smooth. \end{rk}

It is possible to characterize $b$-integrable systems in terms of what we call a $b$-function:

\begin{prop}\label{prop
} Near a regular point of $Z$, any $b$-integrable system on a $b$-symplectic manifold is locally equivalent to a $b$-integrable system of the form $F = (f_1, \ldots, f_n)$, where $f_1$ is a $b$-function and $f_2, \ldots, f_n$ are smooth ($C^\infty$) functions. Moreover, we may assume that $f_1 = \log|t|$, where $t$ is a global defining function for $Z$. \end{prop}

In the case of $b$-functions  the Darboux-Carathéodory theorem \emph {locally} extends a set of $n$ Poisson commuting and functionally independent $b$-functions to a ``$b$-Darboux'' coordinate system. This result is proven in \cite{KMS16}:

\begin{thm}[Kiesenhofer-Miranda-Scott, \cite{KMS16}]\label{th:darboux-caratheodory}
 Let $(M^{2n}, \omega)$ be a $b$-symplectic manifold, $m$ be a point on the exceptional hypersurface $Z$, and $f_1,\ldots,f_{n}$ be $b$-functions, defined on a neighbourhood of $m$, with the following properties
 \begin{itemize}
 \item $f_1,\ldots,f_{n}$  Poisson commute
 \item  $X_{f_1}\wedge \dots \wedge X_{f_n}$
  is a nonzero section of $\Lambda^n(^bTM)$ at $m$. 
 \end{itemize}   
  Then there exist $b$-functions $(g_1, \dots, g_n)$ around $m$ such that
\[ \omega = \sum_{i=1}^{n} df_i \wedge dg_i.\]
and the vector fields $\{X_{f_i}, X_{g_j}\}_{i, j}$ commute.

Moreover, if $f_1$ is not a smooth function, i.e. $f_1 = c\log|t|$ for some $c \neq 0$ and some local defining function $t$ of $Z$, then the functions $g_i$ can be chosen to be smooth functions for which 
$$(t, f_2, \dots, f_{n},  g_1, \dots, g_n)$$
 are local coordinates around $m$.
\end{thm}

\begin{rk}
    The proof of this theorem connects the construction of commutative frames with the existence of commuting vector fields determined by the first integrals of the integrable system. It would be valuable to connect this to the integration methods in \cite{CF03}.
\end{rk}

\section*{Appendix: Commutative frames of a Lie algebroid}
\label{subsection:holgrpd}

This appendix is key for sections \ref{sec.hol.grpd}, \ref{sec.int.ar} and \ref{sec.int.e} of this article. 

Here we want to describe explicitly the groupoid composition and inverse for the holonomy and source-simply connected groupoid of a foliation in special charts.

Composition on a Lie groupoid is often challenging to express. However, when it satisfies a condition akin to commutativity, the problem becomes easier. We describe groupoids whose Lie algebroid meets this commutative condition, meaning it has a commutative frame, as defined below.

\begin{defn}\label{def:comfrm}
Let $A\fto M$ be a Lie algebroid and $p\in M$. There is a commutative frame around $p$ if there exists a neighbourhood $U$ of $p$ and linearly independent sections $X_1,\dots,X_k\in \gs(A|_U)$ generating $\gs(A|_U)$ such that $[X_i,X_j]=0$ for all $i,j\leq k$.
\end{defn}

\begin{ex}\hfill
    \begin{itemize}
        \item The tangent bundle $TM$ of a manifold $M$ and any regular foliation on $M$ have commutative frames near any point $p\in M$.
        \item A non-commutative finite-dimensional Lie algebra, viewed as a Lie algebroid over a point, has no commutative frames, because it cannot have commutative basis.
        \item The Lie algebroid in $\KR^2$ associated to the foliation generated by $x\dex$ and $\dey$ has commutative frames near any point $p\in \KR^2$.
        \item The Lie algebroid in $\KR^2$ associated to the foliation generated by $x\dex$ and $x\dey$ does not have a commutative frame near the point $(0,0)$.
        
    \end{itemize}
\end{ex}
Given any commutative frame $X_1, \dots, X_k$ on $\gs(A)$ near a point $p$, there exists a small neighbourhood $U$ of $p$ and an open set $\CU' \subset \mathbb{R}^k \times U$ containing $(0,p)$ such that the following defines a source-simply connected Lie groupoid: \begin{itemize} \item The arrows are points in $\CU'$, and the objects are points in $U' := {\rm Pr}_U(\CU') \subset U$. \item The source is $\bs = {\rm Pr}_U$, and the target of an element $(v_1, \dots, v_k, u) \in \CU' \subset \mathbb{R}^k \times U$ is given by the flow $\bt((v_1, \dots, v_k, u)) := \Phi^{v_1 \rho(X_1) + \dots + v_k \rho(X_k)}_1 (u)$. \item Composition is given by addition in $\mathbb{R}^k$, the inverse by multiplying by $-1$, and the identity by the zero section. \end{itemize}

The Lie algebroid of the groupoid above is $A'=\KR^k\times U'$ with anchor $\rho\colon A'\fto TU'$ given by the formula $\rho(e_i,u)=\rho(X_i(u))$ for $e_1,\dots,e_k$ the canonical basis of $\KR^k$ and $u\in U'$. This Lie algebroid has commutative frames near any point $p\in U'$ and $A|_{U'}\cong A'$.

\begin{prop}\label{prop:comm.chart.grpd}
Let $A \fto M$ be an ai-algebroid with a commutative frame around $p \in M$, and let $\CG$ be a Lie groupoid integrating $A$. There exists a local diffeomorphism around $\bie(p)$, preserving composition, inverse, and identity, between $\CG$ and the groupoid $\CU'$ described above.
\end{prop}
\begin{proof}(Sketch) The Lie algebroid $A|_{U'}$ is isomorphic to the Lie algebroid $A' := \mathbb{R}^k \times U' \to U'$, which is the Lie algebroid of $\CU'$. By Lie's second theorem, there exists a Lie groupoid map $\CU' \to \CG$. This map is a local diffeomorphism around $\bie(p)$ because its derivative is an isomorphism, corresponding to the isomorphism $A' \cong A|_{U'}$. \end{proof}

The proposition above provides charts for $\CG$ and a local description of its groupoid structure near $\iota(p)$.

The following proposition gives a necessary (but likely not sufficient) condition for the existence of a commutative frame around a point.

\begin{prop} Let $A \to M$ be a Lie algebroid and $p \in M$. If the isotropy Lie algebra $\g_p := \ker(\rho)_p$ is non-commutative, then there is no commutative frame around $p$. \end{prop}

The tangent, $b^m$-tangent, $c$-tangent bundles, and regular foliations have commutative frames because in its definition they are described by the existence of a commutative frame, ex: for $b^m$ there is $\left<x_0^m \de_{x_0},\de_{x_1},\dots\right>$; for a case of $c$ there is $\left<x_0 \de_{x_0},x_1\de_{x_1},\de_{x_2},\dots \right>$ or the local form for regular foliations.

\begin{lem}\label{lem.regular.pair}
    If $A$ is a regular foliation, i.e. the anchor map $\rho\colon A\fto TM$ is injective, $\CG$ a groupoid integrating $A$ then the map:
    $$\bt\times\bs: \CU''\fto M\times M$$
    is a groupoid immersion (groupoid map, smooth, and with injective derivative).

    This means that, near the identity bisection, the holonomy groupoid of a regular foliation is isomorphic to a submanifold of the pair groupoid $M\times M$.
\end{lem}
 \begin{proof} Injective Lie algebroids are integrable and have commutative frames. Then, by Proposition \ref{prop:comm.chart.grpd}, the map $\bt \times \bs$ is an immersion (since the Lie algebroid is injective). The neighbourhood $\CU'$ can be shrunk to achieve an embedding if necessary. \end{proof}

\bibliographystyle{alpha}
\bibliography{alga}

 \end{document}